\documentclass[reqno]{article}

\usepackage{amssymb,amsmath,amsthm}
\usepackage{newpxtext,newpxmath} % commentare per tornare ai fonts abituali
\usepackage{mathtools}
\usepackage{amsfonts}
\usepackage{a4wide}
\usepackage[normalem]{ulem}
\usepackage{mathrsfs,enumerate}
\usepackage[bookmarks=false]{hyperref}

\usepackage{color}
\usepackage{pdfsync}

\usepackage[
color,
final]
{showkeys}   

\definecolor{refkey}{gray}{.3}
  \definecolor{labelkey}{gray}{.3}

\renewcommand*\showkeyslabelformat[1]{%
  \kern-8pt\fontsize{6}{2}\selectfont\sffamily#1}

\numberwithin{equation}{section}  % per cambiare la numerazione
\newtheorem{theorem}{Theorem}[section]  % per cambiare la numerazione
\newtheorem{lemma}[theorem]{Lemma}
\newtheorem{corollary}[theorem]{Corollary}
\newtheorem{proposition}[theorem]{Proposition}

\newtheorem{remark}[theorem]{Remark}
\newtheorem{example}[theorem]{Example}

\newtheorem{definition}[theorem]{Definition}

\newcommand{\N}{\mathbb{N}}

\renewcommand{\P}{\mathbb{P}}
\newcommand{\Q}{\mathbb{Q}}
\newcommand{\R}{\mathbb{R}}

\newcommand{\cA}{{\ensuremath{\mathcal A}}}

\newcommand{\cF}{{\ensuremath{\mathcal F}}}

\newcommand{\cK}{{\ensuremath{\mathcal K}}}

\newcommand{\cM}{{\ensuremath{\mathcal M}}}

\newcommand{\cP}{{\ensuremath{\mathcal P}}}

\newcommand{\cS}{{\ensuremath{\mathcal S}}}

\newcommand{\frB}{\mathfrak {B}}

\newcommand{\frF}{\mathfrak {F}}
\newcommand{\frG}{\mathfrak {G}}

\newcommand{\frS}{\mathfrak {S}}

\newcommand{\frH}{\mathfrak {H}}
\newcommand{\frI}{\mathfrak {I}}

\renewcommand{\aa}{{\boldsymbol a}}
\newcommand{\bb}{{\boldsymbol b}}

\newcommand{\ff}{{\boldsymbol f}}
\renewcommand{\gg}{{\boldsymbol g}}

\newcommand{\ii}{{\boldsymbol i}}

\newcommand{\xx}{{\boldsymbol x}}

\newcommand{\yy}{{\boldsymbol y}}
\newcommand{\sfd}{{\mathsf d}}

\newcommand{\sfD}{{\mathsf D}}

\newcommand{\sfX}{{\mathsf X}}
\newcommand{\sfY}{{\mathsf Y}}
\newcommand{\sfZ}{{\mathsf Z}}
\newcommand{\sfW}{{\mathsf W}}

\newcommand{\rmG}{{\mathrm G}}
\newcommand{\rmL}{{\mathrm L}}

\newcommand{\rmT}{{\mathrm T}}

\newcommand{\supp}{\mathop{\rm supp}\nolimits}   % supporto 
\newcommand{\restr}[1]{\lower3pt\hbox{$|_{#1}$}}
\newcommand{\Restriction}[1]{\lower3pt\hbox{$|_{#1}$}}  %\Restriction{arg}  restrizione ad arg
\newcommand{\eps}{\varepsilon}  
\newcommand{\nchi}{{\raise.3ex\hbox{$\chi$}}}
\renewcommand{\d}{{\mathrm d}}

\newcommand{\Dom}[1]{\ensuremath{\mathrm{Dom}(#1)}}
\newcommand{\Graph}[1]{\ensuremath{\mathrm{Graph}(#1)}}

\newcommand{\lipmp}[1]{{%\mathrm{lip}
\mathfrak {lip}\kern-1pt^\mp\kern-1pt#1}}
\newcommand{\lipm}[1]{{%\mathrm{lip}
\mathfrak {lip}\kern-1pt^-\kern-1pt#1}}
\newcommand{\lipp}[1]{{%\mathrm{lip}
\mathfrak {lip}\kern-1pt^+\kern-1pt#1}}

\newcommand{\cPg}{\cP^{\rm det}}

\newcommand{\CCmaps}[2]{\operatorname{CG}(#1;#2)}

\newcommand{\tCC}{{\upshape CG}\relax }

\newcommand{\mres}{\mathbin{\vrule height 1.6ex depth 0pt width
0.13ex\vrule height 0.13ex depth 0pt width 1.3ex}}

\begin{document}
 \title{Continuous transformations of probability measures and their transport representations}
 \author{Hugo Lavenant\footnote{Bocconi University, Department of Decision Sciences and BIDSA, Milan, (Italy). Email:
\textsf{hugo.lavenant@unibocconi.it}}~~and Giuseppe Savaré\footnote{Bocconi University, Department of Decision Sciences and BIDSA, Milan (Italy). Email: \textsf{giuseppe.savare@unibocconi.it}
Partially supported by funding from the European
Research Council (ERC) under the European Union’s Horizon Europe
research and innovation programme (grant agreement No. 101200514,
project acronym OPTiMiSE).
 Views and opinions expressed are however those of the author(s) only and do not necessarily reflect those of the European Union or the European Research Council Executive Agency. Neither the European Union nor the granting authority can be held responsible for them. 
}}
\date{\today}
\maketitle
\begin{abstract}
Given a function $F$ transforming a probability measure $\mu$ into another one $F(\mu)$, we study the existence and regularity of a transport representation of it. That is, we ask whether we can represent the image $F(\mu)$ of the input probability measure $\mu$ as the push-forward of $\mu$ by a map $\ff(\cdot,\mu)$ which may depend on $\mu$; and furthermore, how regular $\ff$ can be chosen depending on $F$. 
Even if $F$ is continuous and a transport representative exists, it cannot necessarily be chosen in a continuous way; however, if $F$ is Lipschitz continuous with respect to the Wasserstein distance, then $\ff$ can be chosen continuous. We provide several examples to illustrate the sharpness of our assumptions.  
This question is motivated by approximation results for transformations of probability distributions with transformers.
\end{abstract}

\section{Introduction}

We analyze transformations of probability measures: given two Polish spaces $\sfX$ and $\sfY$, we consider a map $F : \cP(\sfX) \to \cP(\sfY)$ which transforms a probability distribution $\mu \in \cP(\sfX)$ into another probability distribution $F(\mu) \in \cP(\sfY)$. We are interested in a transport representative $\ff : \sfX \times \cP(\sfX) \to \sfY$ such that $\ff(\cdot, \mu) : \sfX \to \sfY$ pushes $\mu$ onto $F(\mu)$. In equation, we want to have 
\begin{equation}
\label{eq:intro_transport_rep}
F(\mu) = \ff(\cdot,\mu)_\sharp \mu \qquad \text{for every } \mu \in \cP(\sfX),
\end{equation}
where $g_\sharp \mu$ is the push-forward of the measure $\mu$ by the map $g$, defined by $g_\sharp \mu(A) = \mu(g^{-1}(A))$ for any Borel set $A$. That is, we want $F(\mu)$ to be the push-forward of $\mu$ by a transport map $\ff(\cdot,\mu)$, but the map is allowed to depend on $\mu$. 

While it is clear that, if a transport representative $\ff$ exists, then $F$ should inherit the regularity of $\ff$, in this article we investigate the reverse problem: if $F$ is given, can we find a $\ff$ such that~\eqref{eq:intro_transport_rep} holds, and how regular can $\ff$ be found? We develop a theoretical framework to prove existence of a (continuous) transport representative, while presenting several examples to discuss the sharpness of our assumptions. 

There is an obvious obstruction to the existence of a transport representative. 
The constant map $F(\mu) = \lambda$, where $\lambda$ is any fixed measure which is not a Dirac mass, cannot be represented as in~\eqref{eq:intro_transport_rep}, as no transport map can push a Dirac mass onto something else than a Dirac mass. On $\sfX = \sfY = \R^d$, if $h : \R^d \to \R_+$ is a continuous non-negative function which integrates to $1$ then $F : \mu \mapsto \mu * h$ provides another counterexample, as $F(\delta_x) = h(y-x) \d y$.
More generally, while it is true that an atomless measure can be transported to any measure, the picture is more intricate if $\mu$ has an atomic part. In our study we will take as an assumption that, at least for empirical measures $\mu$, there exists at least one transport map between $\mu$ and $F(\mu)$.

\subsection{Motivations}

\paragraph{Approximation of transformations of measures by transformers}
Recall that neural networks can be seen as parametric classes of functions $(f_\theta : \sfX \to \sfY)_{\theta}$ where the parameter $\theta$ encodes the weights of the network. When $\sfX$ and $\sfY$ are Euclidean spaces, they have the ability, with sufficient parameters, to approximate any reasonable function $f : \sfX \to \sfY$: they are ``universal approximators'', see e.g. the survey \cite{pinkus1999approximation} and references therein. 

Transformers, introduced in \cite{vaswani2017attention}, are neural networks which take as input not a single element of $\sfX$, but a collection $(x_i)_{i=1, \ldots, m}$ whose size $m$ is not fixed a priori. The $x_i$'s, which belong to $\sfX$ are called ``tokens'': they correspond to vectorial embeddings of groups of letters in natural language processing, or patches of an image in vision applications. The transformer processes the whole family of tokens at once and outputs a new collection $(y_i)_{i=1, \ldots, m}$ of tokens. Specifically $y_i = \ff_\theta(x_i, (x_\ell)_{\ell \neq i})$ for some parametric map $\ff_\theta$ which takes as input $x_i$, the token, as well as the ``context'' $(x_\ell)_{\ell \neq i}$, which is the collection of all other tokens. Again here the parameter $\theta$ encodes the weights of the network.

For encoder transformers, the map $\ff_\theta$ is permutation invariant, meaning $\ff_\theta(x_i, (x_\ell)_{\ell \neq i})$ does not depend on the ordering of the context $(x_\ell)_{\ell \neq i}$. Given this invariance it is natural to encode the collection of tokens as an empirical measure $\mu = \frac{1}{m} \sum_{i=1}^m \delta_{x_i} \in \cP(\sfX)$ and write $y_i = \ff_\theta(x_i, \mu)$, so that the transformer is a parametric map from $\sfX \times \cP(\sfX)$ to $\sfY$, with the second argument restricted to being an empirical measure. With $\nu = \frac{1}{m} \sum_{i=1}^m \delta_{y_i} \in \cP(\sfY)$, we have $\nu = \ff_\theta(\cdot,\mu)_\sharp \mu$ so that the transformer architecture induces a map $F : \cP(\sfX) \to \cP(\sfY)$ defined on empirical measures of the form~\eqref{eq:intro_transport_rep}.
This point of view has been put forward in \cite{vuckovic2020mathematical,sander2022sinkformers} and is especially fruitful to handle the mean-field limit, when the number of token diverges; see also \cite{geshkovski2025mathematical}. 

In this context one can also study the universal approximation properties of these neural networks. For families of tokens with given size, we refer in this direction to \cite{yun2019transformers,alberti2023sumformer}. Seeing transformers as inducing transformations at the level of probability measures, \cite{geshkovski2024measure} showed that, given a finite collection $(\mu_j,\nu_j)_j$ of input/output measures, under reasonable assumptions, one can build a transformer architecture $\ff_\theta$ such that $\ff_\theta(\cdot,\mu_j)_\sharp \mu_j \simeq \nu_j$ for all $j$. One of their assumption is the existence of a transport map, depending on $j$, sending $\mu_j$ onto $\nu_j$ for all $j$. 
The work \cite{furuya2024transformers} rather studied the problem of approximating a map from $\sfX \times \cP(\sfX)$ to$ \cP(\sfY)$: they show that, in a suitable sense, any jointly continuous map $\ff : \sfX \times \cP(\sfX) \to \sfY$ can be approximated by a transformer architecture. 
Thus any map $F : \cP(\sfX) \to \cP(\sfY)$ which admits a jointly continuous transport representative can be approximated with a transformer architecture globally over $\cP(\sfX)$: if $\ff$ is the transport representative of $F$, we simply approximate $\ff$ with a transformer. 
Hence the question which is the starting point of our investigation: which are the maps $F : \cP(\sfX) \to \cP(\sfY)$ which admit a (continuous) transport representative $\ff : \sfX \times \cP(\sfX) \to \sfY$? 

\paragraph{Lagrangian representations of evolutions in Wasserstein spaces}
The present study is also related to evolutions in the Wasserstein space of probability measures. In \cite{cavagnari2023lagrangian}, evolutions of probability measures leading to a semi-group $(F_t)_{t \geq 0}$ of maps $\cP(\sfX) \to \cP(\sfX)$ satisfying $F_{t+s} = F_t \circ F_s$ are considered. Under a structural assumption on the generator called \emph{total dissipativity}, a representation of the semigroup as $F_t(\mu) = f_t(\cdot,\mu)_\sharp \mu$, similar to~\eqref{eq:intro_transport_rep}, is proved. This evolution can also be described as a Lagrangian evolution, that is, with a family $\hat{\frF}_t$ of transformations of a $\sfX$-valued random variable $X$ into another $\sfX$-valued random variable $\hat{\frF}_t(X)$ such that, if the law of $X$ is $\mu$, then the law of $\hat{\frF}_t(X)$ is $F_t(\mu)$.

In the present work, we drop the dependence on the temporal variable and only study the structure of maps satisfying~\eqref{eq:intro_transport_rep}. Following \cite{cavagnari2023lagrangian} we also make the link with Lagrangian representations, that is, representations of transformations of measures as transformations of random variables.   

\subsection{A snapshot of the results}

We need a topological setting to discuss the continuity of $F$ and its transport representative, we refer to Section~\ref{sec:prelim} for the full picture and some reminders. Our theory covers the setting $\sfX = \R^{d_1}$ and $\sfY = \R^{d_2}$ (possibly with $d_1 \neq d_2$), and more generally any Polish space endowed with its Borel $\sigma$-algebra. The spaces $\cP(\sfX)$ and $\cP(\sfY)$ of Borel probability measures on respectively $\sfX$ and $\sfY$ are endowed with the topology of narrow convergence, generated by the duality pairing with bounded continuous function, and the associated Borel $\sigma$-algebra.

As already commented above, not every map has a transport representative. Indeed,
if $F : \cP(\sfX) \to \cP(\sfY)$ has a transport representative, then necessarily we have a non-splitting property: atoms do not split, that is, a (finite) combination of Dirac masses cannot be split. We will take it as the minimal assumption needed to make our theory work. We call $\mu \in \cP(\sfX)$ an empirical measure if it is a finite sum of Dirac masses with rational coefficients.

\begin{definition}
We say that $F : \cP(\sfX) \to \cP(\sfY)$ is non-splitting on empirical measures if: for any measure $\mu = \sum_{i=1}^m a_i \delta_{x_i} \in \cP(\sfX)$ with $x_1, \ldots, x_m$ distinct elements in $\sfX$ and $a_1, \ldots, a_m \in \mathbb{Q}_+$ for all $i$, there exist $y_1, \ldots, y_m \in \sfY$ (not necessarily distinct) such that 
\begin{equation}
\label{eq:tran_rep_discrete}
F(\mu) = F \left( \sum_{i=1}^m a_i \delta_{x_i} \right) = \sum_{i=1}^m a_i \delta_{y_i}. 
\end{equation}
\end{definition}

Equivalently, for any empirical measure $\mu$, we assume that there exists $f : \sfX \to \sfY$ such that $F(\mu) = f_\sharp \mu$.

\begin{remark}
\label{rmk:ns_NP}
Checking that $F$ is non-splitting on empirical measures can be, algorithmically speaking, a delicate task. Consider $\mu = \sum_{i=1}^m a_i \delta_{x_i}$ with the $a_i$ rational and the $x_i$ all distinct, and $\nu = \frac{1}{2}(\delta_{y_1} + \delta_{y_2})$ with $y_1 \neq y_2$. Determining if there exists $f$ with $f_\sharp \mu = \nu$ means finding a partition $(I_1,I_2)$ of $\{ 1,\ldots,m \}$ such that $\sum_{i \in I_1}  a_i = \sum_{i \in I_2} a_i$. This is an instance of the partition problem, which is NP-complete \cite{garey1979computers}. 
\end{remark}

We will not focus on conditions guaranteeing the non-splitting property. Rather, assuming that it holds, we want to understand if a transport representative which depends continuously on the measure $\mu$ can be constructed.  
In this context our first main result reads as follows.

\begin{theorem}
\label{thm:main_continuity}
For Polish spaces $\sfX, \sfY$, assume $F : \cP(\sfX) \to \cP(\sfY)$ is non-splitting on empirical measures and continuous. 

Then there exists $\ff : \sfX \times \cP(\sfX) \to \sfY$ Borel such that $F(\mu) = \ff(\cdot, \mu)_\sharp \mu$ for all $\mu \in \cP(\sfX)$.
\end{theorem}

There are two notable consequences: that the non-splitting property on empirical measures guarantees existence of the transport representative $\ff$ on \emph{all} measures; and that $\ff$ can be chosen as a Borel map. However, the transport representative $\ff(x,\mu)$ may not depend continuously on $\mu$ (see Example~\ref{ex:t_rep_non_c}), or may not depend continuously on $x$ (see Example~\ref{ex:continuity_x}). 

Next, we show that if we assume a more quantitative version of continuity, we can obtain much stronger consequences. For an exponent $p \in [1, + \infty)$, we denote by $W_p$ the $p$-Wasserstein distance, whose precise definition can be found in Section~\ref{sec:prelim}, it is defined on $\cP_p(\sfX)$ and $\cP_p(\sfY)$ the space of probability distributions with finite $p$ moments. The definition of length space is recalled in Section~\ref{sec:main_results}, and $\supp(\mu)$ denotes the topological support of a measure $\mu$.

\begin{theorem}
\label{thm:main_Lipschitz}
Let $\sfX$ a Polish length space and $\sfY$ a Polish space.
Assume $F : \cP_p(\sfX) \to \cP_p(\sfY)$ is non-splitting on empirical measures and Lipschitz continuous, meaning there exists $L \geq 0$ such that $W_p(F(\mu),F(\mu')) \leq L W_p(\mu,\mu')$ for all $\mu,\mu' \in \cP_p(\sfX)$.

Then there exists $\ff : \sfX \times \cP(\sfX) \to \sfY$ Borel such that $F(\mu) = \ff(\cdot, \mu)_\sharp \mu$ for all $\mu \in \cP(\sfX)$, and which is jointly continuous in the sense: for every sequences $(x_n)_{n \in \N}$ and $(\mu_n)_{n \in \N}$,
\begin{equation*}
x_n \in \supp(\mu_n),\; x_n \to x \text{ in } \sfX,\; \mu_n \to \mu \text{ in } \cP_p(\sfX),\; x \in \supp(\mu) \;\Rightarrow\; \ff(x_n,\mu_n) \to \ff(x,\mu) \text{ in } \sfY.
\end{equation*}
Moreover, the restriction of $\ff(\cdot,\mu)$ to $\supp(\mu)$ is uniquely determined and $L$-Lipschitz.
\end{theorem}

Compared to the previous result, we obtain a transport representative $\ff(x,\mu)$ which is continuous in both $x$ and $\mu$. While continuity of $\ff$ in $\mu$ could be expected from the one of $F$, continuity in $x$ is more surprising as we did not assume it directly. Continuity holds only at $(x,\mu)$ with $x \in \supp(\mu)$, and this condition is sharp (see Example~\ref{ex:lift_discontinuous_supp}). The length space assumption can be relaxed: it is enough that the intrinsic length distance on $\sfX$ is bi-Lipschitz equivalent to $\sfd_\sfX$, see Remark~\ref{rmk:length_space}.%

Building on Theorem~\ref{thm:main_Lipschitz}, we also derive the following approximation result for the transformation $F$, proved in Section~\ref{sec:approximation}.

\begin{theorem}
\label{thm:approximation}
Let $\sfX$ be a Polish length space and $\sfY$ a Polish space. Assume $F : \cP_p(\sfX) \to \cP_p(\sfY)$ is non-splitting on empirical measures and $L$-Lipschitz. Let $\cA$ be a family of mappings $\gg : \sfX \times \cP_p(\sfX) \to \sfY$, and assume that the uniform closure of $\cA$ contains every jointly continuous mapping $\sfX \times \cP_p(\sfX) \to \sfY$ which is $L$-Lipschitz in the first variable.

Then for every $\varepsilon > 0$ there exists $\gg \in \cA$ such that
\begin{equation*}
W_p(\gg(\cdot,\mu)_\sharp \mu, \, F(\mu)) < \varepsilon \qquad \text{for every } \mu \in \cP_p(\sfX).
\end{equation*}
\end{theorem}

\begin{remark}[Connection with transformer universality]
\label{rmk:transformers}
The assumption on $\cA$ in Theorem~\ref{thm:approximation} is tailored to applications where $\cA$ is a parametric family of architectures whose uniform closure is rich enough to contain every jointly continuous $L$-Lipschitz mapping. A natural instance is the case of transformer architectures: the universality result of \cite{furuya2024transformers} shows that, on a compact domain $\sfX\subset \R^d$, the set of deep transformer mappings is uniformly dense in the space of jointly continuous functions $\sfX \times \cP(\sfX) \to \sfY$. In particular it contains every $L$-Lipschitz (in the first variable) jointly continuous mapping in its uniform closure, so Theorem~\ref{thm:approximation} applies: every $L$-Lipschitz non-splitting transformation $F$ can be approximated, in Wasserstein distance and uniformly in $\mu$, by the action of a deep transformer.
\end{remark}

In the rest of this article we develop a framework to prove these results, as well as providing several additional results which could be interesting in themselves, including a link with Lagrangian representations of continuous transformations of probability measures. An interesting point that we want to emphasize already is how we discuss in general the measurability and continuity of the transport representative $\ff$. 
Indeed the condition $\ff(\cdot,\mu)_\sharp \mu = F(\mu)$ only involves the values of $\ff(\cdot,\mu)$ $\mu$-a.e. Thus the most we can hope is to define without ambiguity $\ff(\cdot,\mu)$ on $\supp(\mu) \subset \sfX$, but the latter set changes with $\mu$. For instance in~\eqref{eq:tran_rep_discrete} we can define $\ff(x_i,\mu) = y_i$, but there is no canonical way to define $\ff(\cdot,\mu)$ outside of $\{ x_1, \ldots, x_m \}$. A more principled way is to see $\ff(\cdot,\mu)$ as an element of $L^p(\mu;\sfX, \sfY)$, that is, an equivalence class of maps $\sfX \to \sfY$, quotiented by the equality $\mu$-a.e. Thus we think at $\mu \mapsto \ff(\cdot,\mu)$ as a map from $\cP_p(\sfX)$ to $L^p(\mu;\sfX; \sfY)$. We still have the problem that the image space $L^p(\mu;\sfX; \sfY)$ depends on $\mu$. 
To be able to talk about continuity of $\mu \mapsto \ff(\cdot,\mu)$, we rely on the \emph{Transport-Lebesgue distance} introduced first in \cite{GarciaTrillos-Slepcev16} for the analysis of variational problems on graphs and their continuous limit, further studied in \cite{thorpe2017transportation,trillos2018variational}. It is a distance on the space of pairs $(\mu,f)$ with $\mu \in \cP_p(\sfX)$ and $f \in L^p(\mu;\sfX; \sfY)$, denoted by $\rmT\rmL_p(\sfX;\sfY)$. The idea is to embed a pair $(\mu,f)$ as a probability distribution in $\sfX \times \sfY$ concentrated on its graph: we consider $\iota(\mu,f) = (\ii \times f)_\sharp \mu$ which maps $(\mu,f) \in \rmT\rmL_p(\sfX;\sfY)$ into  a probability distribution over $\sfX \times \sfY$ concentrated on the graph of $f$. The $\rmT\rmL_p$ distance between $(\mu,f)$ and $(\mu',f')$ is nothing else than the classical Wasserstein distance between probability distributions $\iota(\mu,f)$ and $\iota(\mu',f')$ on $\cP_p(\sfX \times \sfY)$. A transport representative $\ff : \sfX \times \cP_p(\sfX) \to \cP_p(\sfY)$ will be said Borel or continuous if the map $\mu \mapsto (\mu,\ff(\cdot,\mu))$ from $\cP_p(\sfX)$ to $\rmT\rmL_p(\sfX;\sfY)$ is Borel, or continuous. We connect back these notions of Borel measurability and continuity with Borel measurability or continuity of $\ff$ as a function of $x$ and $\mu$ to obtain the statements of Theorem~\ref{thm:main_continuity} and Theorem~\ref{thm:main_Lipschitz}, see Remark~\ref{rmk:mes_cont} below.

\subsection{Organization of the article}

The proof of our main results, Theorems~\ref{thm:main_continuity},~\ref{thm:main_Lipschitz}, and~\ref{thm:approximation} is done in Section~\ref{sec:main_results}. But before proceeding to it we will need to derive several results which may be independent on their own right.

We first recall some notations and basic results of measure theory in Section~\ref{sec:prelim}.

In Section~\ref{sec:TLp}, we study the space $\rmT\rmL_p(\sfX;\sfY)$ made of couples $(\mu,f)$ with $f \in L^p(\mu;\sfX;\sfY)$, introduced in \cite{GarciaTrillos-Slepcev16}. Though several important results were already discovered in \cite{GarciaTrillos-Slepcev16}, we prove additional properties of this space and in particular identify a criterion of relative compactness which plays in crucial role in the proof of our main result and may be of independent interest. 

As the reader can see, discrete measures play a distinguished role in our study. Rather than focusing on measures with equal weights, we actually need measures whose weights are as different as possible, as they give the possibility to identify unequivocally the transport representative. We call these discrete measures \emph{generic}, and we introduce and study them in Section~\ref{sec:generic_measures}.

In Section~\ref{sec:main_results} we actually prove our main results. We investigate what we can say on the transport representative $\ff$ depending if $F$ is Borel, continuous or Lipschitz. We provide several examples to discuss the sharpness of our assumptions.   

Eventually in Section~\ref{sec:lagrangian} we make the link with representations of $F$ as a transformation of random variables, the so-called Lagrangian representations following \cite{cavagnari2023lagrangian,CSS25}. We prove that having a Lagrangian representation invariant by measure-preserving isomorphisms is equivalent to the existence of a continuous transport representative.

\subsection*{Acknowledgments}

We warmly thank Gabriel Peyré for raising the question which was the starting point of our investigation and for fruitful discussions. 

\section{Preliminaries}
\label{sec:prelim}

\subsection{Notations and basic definitions of measure theory}

We refer to \cite{AGS} for definitions and results of measure theory and optimal transport, we highlight here the results we will need. All our topological spaces are metric spaces, endowed with their Borel $\sigma$-algebra.

We denote by $\sfX,\sfY$ complete separable metric spaces endowed with respective metrics $\sfd_\sfX$ and 
$\sfd_\sfY$. They are endowed with their Borel $\sigma$-algebra. We work with an exponent $p \in [1, + \infty)$ fixed in the work, and the product space $\sfX \times \sfY$ is equipped with the $\ell_p$ distance 
\begin{equation*}
\sfd_{\sfX \times \sfY}((x,y),(x',y')) = \left( \sfd_\sfX(x,x')^p + \sfd_\sfY(y,y')^p \right)^{1/p}.
\end{equation*}

We denote by $\cM(\sfX)$  the space of 
finite and positive Borel measures on $\sfX$;
$\cP(\sfX)$ is the subspace of probability measures in $\cM(\sfX)$. The support of a measure $\mu$, written $\supp(\mu)$, is the smallest closed set containing all the mass of $\mu$. It is characterized as: a point $x \in \sfX$ belongs to $\supp(\mu)$ if and only if $\mu(B(x,r)) > 0$ for all $r > 0$, being $B(x,r)$ the open ball of center $x$ and radius $r$. If $A$ is a Borel set and $\mu \in \cM(\sfX)$, $\mu \mres A$ denotes the restriction of the measure $\mu$ to $A$, defined by $(\mu \mres A)(B) = \mu(A \cap B)$.

The space $\cP(\sfX)$ is endowed with the topology of narrow convergence (a.k.a. weak convergence) generated by the duality pairing with continuous and bounded functions. A subset $\cS \subseteq \cP(\sfX)$ is relatively compact if and only if is tight, meaning that for every $\varepsilon  >0$ we can find a compact $\cK$ of $\sfX$ such that $\mu(\cK) \geq 1 - \varepsilon$ for all $\mu \in \cS$.  

Given Polish spaces $\sfX$, $\sfY$, $\mu \in \cP(\sfX)$ and $\nu \in \cP(\sfY)$ we denote by $\Gamma(\mu,\nu)$ the subset of $\cP(\sfX \times \sfY)$ made of probability measures on the product space whose marginals are $\mu$ and $\nu$. Given $\gamma \in \cP(\sfX \times \sfY)$ a probability distribution on a product space, we can consider its disintegration with respect to its first marginal: with $\mu \in \cP(\sfX)$ the first marginal of $\gamma$, it is a family $(\gamma_x)_{x \in \sfX}$ of probability distributions over $\sfY$, unique $\mu$-a.e., and measurable (in the sense $(x,A) \mapsto \gamma_x(A)$ defines a Markov kernel) such that: 
\begin{equation*}
\gamma = \int_{\sfX} (\delta_x \otimes \gamma_x) \, \d \mu(x)
\end{equation*} 

We denote by $\cP_p(\sfX)$ the space of probability measures $\mu$ such that $\int \sfd_\sfX(x_0,x)^p \d \mu(x) < + \infty$ for some $x_0 \in \sfX$. It is endowed with the $p$-Wasserstein distance: if $\mu,\mu' \in \cP_p(\sfX)$, 
\begin{equation*}
W_p(\mu,\mu') = \left( \min_{\eta \in \Gamma(\mu,\nu)} \int_{\sfX \times \sfX} \sfd_\sfX(x,x')^p \, \d \eta(x,x') \right)^{1/p}.
\end{equation*}
The infimum is always reached, and any optimal $\eta$ is called an \emph{optimal transport plan} between $\mu$ and $\nu$.
On $\cP_p(\sfX)$, the quantity $W_p$ defines a distance metrizing narrow convergence together with convergence of $p$-moments. Specifically, a function $\phi : \sfX \to \R$ has $p$-growth if $|\phi(x)| \leq C(\sfd_\sfX(x_0,x)^p+1)$ for some $C \in \R$ and $x_0 \in \sfX$. Then it can be proved that for any sequence $(\mu_n)_{n \in \N}$ and any $\mu$,
\begin{equation*}
W_p(\mu_n,\mu) \to 0 \quad \Leftrightarrow \quad \int_\sfX \phi \, \d \mu_n \to \int_\sfX \phi \, \d \mu \; \text{ for every } \phi : \sfX \to \R \text{ continuous with } p \text{-growth.}
\end{equation*}
We always endowed $\cP_p(\sfX)$ with the topology induced by $W_p$. Its $\sigma$-algebra is generated by the maps $\mu \mapsto \mu(A)$ for $A$ Borel. In the case $\sfd_\sfX$ is bounded we recover the narrow convergence of probability measures. We will use Wasserstein distances over the spaces $\cP_p(\sfX)$, $\cP_p(\sfY)$, and $\cP_p(\sfX \times \sfY)$, and always use $W_p$ to refer to the Wasserstein distances over these different spaces, in order to lighten notations.

If $\mu \in \cP(\sfX)$ then we denote by $L^p(\mu;\sfX;\sfY)$ the space of Borel maps $f : \sfX \to \sfY$ which satisfy $\int \sfd_\sfY(y_0,f(x))^p \d \mu(x) < + \infty$ for some $y_0 \in \sfY$, quotiented by the equivalence relation of being equal $\mu$-a.e. It is endowed with the distance 
\begin{equation*}
\sfd_{L^p(\mu;\sfX;\sfY)}(f,f') = \left( \int_{\sfX} \sfd_{\sfY}(f(x),f'(x))^p \, \d \mu(x) \right)^{1/p},
\end{equation*}
which turns it into a complete metric space, and is abbreviated $\sfd_{L^p(\mu)}$ or $\sfd_{L^p(\sfX;\sfY)}$ when the context is clear. 
We write $\ii :\sfX \to \sfX$ for the identity map.

\subsection{Preliminary technical results}

We start by recalling a variant of the Skorokhod representation theorem of convergence in $W_p$. We recall that as $[0,1]$ endowed with its Borel $\sigma$-algebra and the Lebesgue measure is a standard Borel space.

\begin{theorem}
\label{thm:skorokhod_Wp}
Let $(\Omega,\frB,\P)$ a standard Borel space. Then $(\mu_n)_{n \in \N}$ converges to $\mu$ in $\cP_p(\sfZ)$ if and only if we can find random variables $Z_n, Z$ in $L^p(\Omega;\sfZ)$ with ${Z_n}_\sharp \P = \mu_n$, $Z_\sharp \P = \mu$ and such that $(Z_n)_{n \in \N}$ converges to $Z$ in $L^p(\Omega;\sfZ)$
\end{theorem}

\begin{proof}
The converse implication is immediate as $W_p(\mu_n,\mu) \leq \sfd_{L^p(\Omega;\sfZ)}(Z_n,Z)$: we can take $\eta_n = (Z_n,Z)_\sharp \P \in \Gamma(\mu_n,\mu)$. For the direct implication, if $(\mu_n)_{n \in \N}$ converges to $\mu$ in $\cP_p(\sfZ)$ then it converges narrowly and the classical Skorokhod theorem yields the existence of random variables $Z_n, Z : \Omega \to \sfZ$ such that ${Z_n}_\sharp \P = \mu_n$, $Z_\sharp \P = \mu$ and $Z_n \to Z$ almost surely \cite[Theorem 8.5.4]{bogachev2007measure}. As the $(\mu_n)_{n \in \N}$ are uniformly $p$-integrable (thanks to the $W_p$ convergence), the convergence of $Z_n$ to $Z$ also holds in $L^p(\Omega;\sfZ)$ by the Vitali convergence theorem. 
\end{proof}

We also collect elementary results on the stability of $(\mu,f) \mapsto f_\sharp \mu$ for $\mu \in \cP_p(\sfX)$ and $f \in L^p(\mu;\sfX;\sfY)$, in this case $f_\sharp \mu \in \cP_p(\sfY)$. We first recall the following qualitative result, see \cite[Lemma 5.2.1]{AGS} for a proof. 

\begin{lemma}
\label{lm:joint_continuity_fsharpmu}
Assume $\mu_n \to \mu$ narrowly and $f_n : \sfX \to \sfY$ converges uniformly on compact subsets of $\sfX$ to a continuous $f$. Then ${f_n}_\sharp \mu_n \to f_\sharp \mu$ narrowly. 
\end{lemma}

We will also need more quantitative estimates. 
First, if we fix $\mu$ and vary $f$,  
\begin{equation}
\label{eq:Wp_Lp}
W_p(f_\sharp \mu, f'_\sharp \mu) \leq  \sfd_{L^p(\mu)}(f,f'),
\end{equation}
which can be seen as $(f \times f')_\sharp \mu \in \Gamma(f_\sharp \mu, f'_\sharp \mu)$ , see also \cite[Eq. (7.1.6)]{AGS}. On the other hand, if we keep $f$ fixed and vary $\mu$:
\begin{equation}
\label{eq:lift_lipschitz}
W_p(f_\sharp \mu, f_\sharp \mu') \leq L W_p(\mu,\mu') \qquad \text{if } f : \sfX \to \sfY \text{ is } L \text{-Lipschitz.}
\end{equation}
This is obtained as $(f,f)_\sharp \eta \in \Gamma(f_\sharp \mu, f_\sharp \mu')$ as soon as $\eta \in \Gamma(\mu, \mu')$.

A correspondence $G : \sfW \twoheadrightarrow \sfZ$ is a map from $\sfW$ to the subsets of $\sfZ$. If $\sfW$ and $\sfZ$ are metric spaces, then the correspondence $G$ is: upper hemicontinuous if $w_n \to w$ and $z_n \to z$ as $n \to + \infty$ with $z_n \in G(w_n)$ implies  $z \in G(w)$ ; and lower hemicontinuous if for every $z \in G(w)$ and every sequence $(w_n)_{n \in \N}$ converging to $w$, there exists $(z_n)_{n \in \N}$ converging to $z$ with $z_n \in G(w_n)$ for all $n$. With this vocabulary, we state a result on the hemicontinuity of $\Gamma$, already investigated in \cite{bergin1999continuity,ghossoub2021continuity,cavagnari2023lagrangian}.

\begin{lemma}
\label{lm:Gamma_hemic}
The correspondence $(\mu,\nu) \to \Gamma(\mu,\nu)$, from $\cP_p(\sfX) \times \cP_p(\sfY)$ into the subsets of $\cP_p(\sfX \times \sfY)$, is non-empty, convex and compact valued, and hemicontinuous.
\end{lemma}

\begin{proof}
This result has already been proved if the $W_p$ convergence is substituted by the narrow convergence of measures: see \cite[Theorem 1]{bergin1999continuity}, as well as the more recent work \cite[Theorem 1]{ghossoub2021continuity} for a shorter proof.

The result then follows easily from the following remark. If $(\mu_n)_{n \in \N}$ and $(\nu_n)_{n \in \N}$ converge in $W_p$, then they are uniformly $p$-integrable \cite[Proposition 7.1.5]{AGS}, thus any sequence $(\gamma_n)_{n \in \N}$ with $\gamma_n \in \Gamma(\mu_n,\nu_n)$ for all $n$ is also uniformly $p$-integrable \cite[Remark 5.2.3]{AGS}. It implies that, if $(\gamma_n)_{n \in \N}$ converges narrowly to a limit, then it also converges in $W_p$ to the same limit. With this remark \cite[Theorem 1]{bergin1999continuity} implies our result. 
\end{proof}

\section{The space $\rmT \rmL_p(\sfX;\sfY)$ }
\label{sec:TLp}

In our analysis, we need to consider maps defined on different domains, and to compare them. We write $\rmT\rmL_p(\sfX;\sfY)$ for the set of such pairs of functions together with their base measure:
\begin{equation*}
\rmT\rmL_p(\sfX;\sfY) = \{ (\mu,f) \ : \ \mu \in \cP_p(\sfX), \; f \in L^p(\mu;\sfX;\sfY) \}.
\end{equation*}
Recalling $\ii : \sfX \to \sfX$ is the identity map, we write $\iota$ for the embedding of $\rmT\rmL_p(\sfX;\sfY)$ in the set $\cP_p(\sfX \times \sfY)$ of measures on the product space:
\begin{equation*}
\iota(\mu,f) = (\ii \times f)_\sharp \mu.
\end{equation*}
In words, we embed the pair $(\mu,f)$ into a measure on the product space whose first marginal is $\mu$, and concentrated on the graph of $f$.
The embedding $\iota$, valued in $\cP_p(\sfX \times \sfY)$, is injective. Indeed if $\gamma = \iota(\mu,f)$ then necessarily $\mu$ is the first marginal of $\gamma$, while the disintegration of $\gamma$ with respect to its first marginal reads $(\delta_{f(x)})_{x \in \sfX}$ for $\mu$-a.e. $x$. We write $\cPg_p(\sfX \times \sfY)$ for the image of $\iota$, that is, the set of couplings concentrated on graphs of Borel functions: 
\begin{equation*}
\cPg_p(\sfX \times \sfY) = \{ \gamma \in \cP_p(\sfX \times \sfY) \ : \ \exists \mu \in \cP_p(\sfX), f \in L^p(\mu;\sfX;\sfY), \text{ such that } \gamma = (\ii \times f)_\sharp \mu  \}.
\end{equation*}

\begin{definition}
On the set $\rmT\rmL_p(\sfX;\sfY)$ we define $\sfd_{\rmT\rmL_p}$ as the pull-back of the metric $W_p$ on $\cP_p(\sfX \times \sfY)$ by $\iota$: 
\begin{equation*}
\sfd_{\rmT\rmL_p}((\mu,f),(\mu',f')) = W_p( \iota(\mu,f), \iota(\mu',f')  )= W_p( (\ii \times f)_\sharp \mu,  (\ii \times f')_\sharp \mu'  ).    
\end{equation*}
\end{definition}

In this section we analyze the metric structure of $\rmT\rmL_p(\sfX;\sfY)$ which was introduced and studied in \cite{GarciaTrillos-Slepcev16,thorpe2017transportation,trillos2018variational}. Compared to these works, we are in a slightly more general context, as they only consider $\sfX$ to be a subset of a Euclidean space and $\sfY = \R$. Though $\rmT\rmL_p(\sfX;\sfY)$ is not a complete space as already noted in these works, we prove it is still completely metrizable. Moreover, we characterize relatively compact sets via a notion of equidispersivity that we introduce. We also make the link between $\rmT\rmL_p(\sfX;\sfY)$-convergence and the convergence of continuous functions with compact domains via an analogue of Lusin's theorem.

In the sequel, a typical element of the product space $\sfX \times \sfY$ is denoted by $(x,y)$, while a typical element of $\sfX \times \sfY \times \sfX \times \sfY$ is denoted by $(x,y,x',y')$. In particular if $\sigma \in \cP(\sfX \times \sfY \times \sfX \times \sfY )$ while $g$ is a function defined on $\sfX \times \sfY \times \sfX \times \sfY$, we abbreviate
\begin{equation*}
\| g(x,y,x',y') \|_{L^p(\sigma)} = \left( \int_{\sfX \times \sfY \times \sfX \times \sfY} |g(x,y,x',y')|^p \, \d \sigma (x,y,x',y') \right)^{1/p}.
\end{equation*}
For a product space $\sfX \times \sfY$ or $\sfX \times \sfY \times \sfX \times \sfY$ we denote by $\pi^i$ the project on the $i$-th factor, and $\pi^{i,j}$ the projection on the $i$-th and $j$-th factor.

\subsection{Preliminary results on couplings which are concentrated on graphs}

We first need some preliminary results on the structure of the space $\cPg_p(\sfX\times \sfY)$. Recall that a subset $G$ of $\sfX \times \sfY$ is the graph of a function if and only if
\begin{equation}
\label{eq:char_graph}
(x,y), (x,y') \in G \quad \Rightarrow \quad y = y'.
\end{equation}
In the case of couplings $\gamma \in \cP_p(\sfX \times \sfY)$, we obtain the following characterization of the ones concentrated on graphs. 

\begin{lemma}
\label{lm:char_cpg}
A coupling $\gamma \in \cP_p(\sfX \times \sfY)$ belongs to $\cPg_p(\sfX\times \sfY)$ if and only if 
\begin{equation}
\label{eq:characterization_det}
\text{for every } \sigma \in \Gamma(\gamma,\gamma), \qquad  \| \sfd_{\sfX}(x,x') \|_{L^p(\sigma)} = 0 \quad \Rightarrow \quad \| \sfd_{\sfY}(y,y') \|_{L^p(\sigma)} = 0 .  
\end{equation}
\end{lemma}

\begin{proof}
Assume first that $\gamma \in \cPg_p(\sfX \times \sfY)$, we write $\gamma = (\ii \times f)_\sharp \mu$. Then $y = f(x)$ for $(x,y)$-a.e. $\gamma$. Thus, if $\sigma \in \Gamma(\gamma,\gamma)$ and $x = x'$ for a.e. $\sigma$, that is, $\| \sfd_{\sfX}(x,x') \|_{L^p(\sigma)} = 0$, we obtain $y = f(x) = f(x') = y'$ for $\sigma$-a.e. $(x,y,x',y')$ which reads $\| \sfd_{\sfY}(y,y') \|_{L^p(\sigma)} = 0$. 

To prove the converse implication, take $\gamma \in \cP_p(\sfX \times \sfY)$ satisfying~\eqref{eq:characterization_det}. Let $\mu \in \cP_p(\sfX)$ the first marginal of $\gamma$ and $(\gamma_x)_{x \in \sfX}$ its disintegration with respect with its first marginal. Let $\sigma \in \cP_p(\sfX \times \sfY \times \sfX \times \sfY)$ defined by 
\begin{equation*}
\sigma = \int_{\sfX} (\delta_x \otimes \gamma_x \otimes \delta_x \otimes \gamma_x ) \, \d \mu (x).
\end{equation*}
We check easily that $\sigma \in \Gamma(\gamma,\gamma)$ and $x=x'$ for $\sigma$-a.e. $(x,y,x',y')$. On the other hand
\begin{equation*}
\| \sfd_{\sfY}(y,y') \|_{L^p(\sigma)}^p = \int_{\sfX} \left( \int_{\sfY \times \sfY} \sfd_\sfY(y,y')^p \d \gamma_x(y) \d \gamma_x(y') \right) \d \mu(x).
\end{equation*}
As this quantity equates $0$ from~\eqref{eq:characterization_det}, we obtain $\int_{\sfY \times \sfY} \sfd_\sfY(y,y')^p \d \gamma_x(y) \d \gamma_x(y') = 0$ for $\mu$-a.e. $x$. Thus for $\mu$-a.e. $x$ the measure $\gamma_x$ is of the form $\delta_{y}$ for some $y \in \sfY$, that we denote by $y = f(x)$. As $(\gamma_x)_{x \in \sfX}$ defines a Markov kernel, we deduce that $f : \sfX \to \sfY$ is Borel \cite[Lemma 1.14]{kallenberg2017random}. We conclude to $\gamma = (\ii \times f)_\sharp \mu$, and $f \in L^p(\mu;\sfX;\sfY)$ as $\gamma \in \cP_p(\sfX \times \sfY)$.
\end{proof}

\begin{lemma}
\label{lm:couplings_sigma_eta}
Assume $\gamma = \iota(\mu,f)$ and $\gamma' = \iota(\mu',f')$ are two elements of $\cPg(\sfX \times \sfY)$. Then $\sigma \in \Gamma(\gamma,\gamma')$ if and only if $\sigma = (\pi^1\times f \times \pi^2 \times f')_\sharp \eta$ for some $\eta \in \Gamma(\mu,\mu')$. 
\end{lemma}

\begin{proof}
The converse implication is immediate: if $\sigma = (\pi^1\times f \times \pi^2 \times f')_\sharp \eta$ for some $\eta \in \Gamma(\mu,\mu')$ then we can check directly that $\sigma \in \Gamma(\gamma,\gamma')$. 

Conversely, take $\sigma \in \Gamma(\gamma,\gamma')$. Let $\eta = \pi^{1,3}_\sharp \sigma$, necessarily $\eta \in \Gamma(\mu,\mu')$. Moreover, as the first marginal of $\sigma$ is $\gamma$, for $\sigma$-a.e. $(x,y,x',y')$ we have $y =f (x)$, and as the second marginal is $\gamma'$ for $\sigma$-a.e. $(x,y,x',y')$ we have $y' =f' (x')$. If $\phi$ is any test function on $\sfX \times \sfY \times \sfX \times \sfY$ then
\begin{align*}
\int_{\sfX \times \sfY \times \sfX \times \sfY} \phi(x,y,x',y') \, \d \sigma(x,y,x',y') & = \int_{\sfX \times \sfY \times \sfX \times \sfY} \phi(x,f(x),x',f'(x')) \, \d \sigma(x,y,x',y') \\
& = \int_{\sfX \times \sfX} \phi(x,f(x),x',f'(x'))\, \d \eta(x,x'), 
\end{align*}
which means $\sigma = (\pi^1\times f \times \pi^2 \times f')_\sharp \eta$. 
\end{proof}

From the definition of $W_p$ on the space $\cP_p(\sfX \times \sfY)$ and Lemma~\ref{lm:couplings_sigma_eta}, we obtain the following representation for the distance $\sfd_{\rmT\rmL_p}$, which was taken as a definition in \cite{GarciaTrillos-Slepcev16}:
\begin{equation}
\label{eq:expr_TL_p}
\sfd_{\rmT\rmL_p}((\mu,f),(\mu',f')) = \left( \min_{\eta \in \Gamma(\mu,\mu')} \int_{\sfX \times \sfX} \left( \sfd_\sfX(x,x')^p + \sfd_\sfY(f(x),f'(x'))^p \right) \, \d \eta(x,x') \right)^{1/p}. 
\end{equation}

For any $(\mu,f) \in \rmT\rmL_p(\sfX;\sfY)$, we can choose a representative of $f \in \rmL^p(\mu;\sfX;\sfY)$ which is Borel as a function of $x$. For our purposes it will be important to select this representative also in a Borel way as we vary $\mu$. This can be achieved as a corollary of a universal disintegration theorem.

\begin{proposition}
\label{prop:represent_TLp}
There exists $\frS : \sfX \times  \rmT\rmL_p(\sfX;\sfY) \to \sfY$ Borel such that, for any $(\mu,f) \in \rmT\rmL_p(\sfX;\sfY)$, there holds $f(x) = \frS(x,(\mu,f))$ for $\mu$-a.e. $x$. 
\end{proposition}

\begin{proof}
From \cite[Corollary 1.26]{kallenberg2017random}, there exists a Borel map $\frI : \sfX \times \cP(\sfX \times \sfY) \to \cP(\sfY)$ such that, for all $\gamma \in \cP(\sfX \times \sfY)$, the family $(\frI(x,\gamma))_{x \in \sfX}$ coincides $\pi^1_\sharp \gamma$-a.e. with the disintegration of $\gamma$ with respect to its first variable. Moreover, as $\{ \delta_y \ : \ y \in \sfY \}$ is closed it is easy to construct a Borel map $\frH : \cP(\sfY) \to \sfY$ such that $\frH(\delta_y) = y$ for any $y \in \sfY$. We conclude by defining $\frS(x,(\mu,f)) = \frH( \frI(x, \iota(\mu,f) ))$. % The conclusion follows as the Borel $\sigma$-algebra of $\cP_p(\sfX)$ is the same as the one of $\cP(\sfX)$. 
\end{proof}

\subsection{The metric structure of $\rmT\rmL_p(\sfX;\sfY)$}

\begin{theorem}
\label{thm:tl_polish}
The space $(\rmT\rmL_p(\sfX;\sfY),\sfd_{\rmT\rmL_p})$ is a separable metric space. 
It is completely metrizable: there exists a metric $\tilde{\sfd}_{\rmT\rmL_p}$ on $\rmT\rmL_p(\sfX;\sfY)$ which generates the same topology as $\sfd_{\rmT\rmL_p}$ and such that $(\rmT\rmL_p(\sfX;\sfY),\tilde{\sfd}_{\rmT\rmL_p})$ is complete.
\end{theorem}

\begin{proof}
The first part is immediate as the embedding of $\rmT\rmL_p(\sfX;\sfY)$ into $\cP_p(\sfX \times \sfY)$ is injective, and by definition, an isometry. Moreover $\cP_p(\sfX \times \sfY)$ is separable, thus $\rmT\rmL_p(\sfX;\sfY)$ too.

Given how we defined $\rmT\rmL_p(\sfX;\sfY)$, it is completely metrizable only if $\cPg_p(\sfX \times \sfY)$ is completely metrizable.
Thus it is enough to prove that $\cPg_p(\sfX \times \sfY)$ is a $G_\delta$, as a $G_\delta$ of a Polish space admits a metric which makes the space complete \cite[Lemma 3.34]{aliprantis2006infinite}. For a fixed $m \geq 1$, consider
\begin{equation}
\label{eq:def_Hm}
H_m = \left\{ \gamma \in \cP_p(\sfX \times \sfY) \ :  \; \exists \sigma \in \Gamma(\gamma,\gamma), \, \| \sfd_{\sfX}(x,x') \|_{L^p(\sigma)} = 0 \text{ and } \| \sfd_{\sfY}(y,y') \|_{L^p(\sigma)} \geq \frac{1}{m}  \right\}. 
\end{equation}
The set $H_m$ is closed: if $(\gamma_n)_{n \in \N}$ is a converging sequence in $H_m$, together with an associated $\sigma_n \in \Gamma(\gamma_n,\gamma_n)$ such that $\| \sfd_{\sfX}(x,x') \|_{L^p(\sigma_n)} = 0$ and $ \| \sfd_{\sfY}(y,y') \|_{L^p(\sigma_n)} \geq \frac{1}{m}$ for all $n$, then by hemicontinuity and compactness of $\Gamma$, see Lemma~\ref{lm:Gamma_hemic}, up to extraction $(\sigma_n)_{n \in \N}$ converges to $\sigma \in \Gamma(\gamma,\gamma)$ in $W_p$ and we can pass to  the limit: $\| \sfd_{\sfX}(x,x') \|_{L^p(\sigma)} = 0$ and $ \| \sfd_{\sfY}(y,y') \|_{L^p(\sigma)} \geq \frac{1}{m}$. 
Moreover by Lemma~\ref{lm:char_cpg} we see that 
\begin{equation*}
\bigcup_{m \ge 1} H_m = \cP_p(\sfX \times \sfY) \setminus \cPg_p(\sfX \times \sfY).  
\end{equation*}
That is, the complement of $\cPg_p(\sfX \times \sfY)$ is a countable union of closed sets, thus $\cPg_p(\sfX \times \sfY)$ is a countable intersection of open sets, that is, a $G_\delta$.   
\end{proof}

\begin{remark}
Note that $(\rmT\rmL_p(\sfX;\sfY),\sfd_{\rmT\rmL_p})$ is complete if and only if $\cPg_p(\sfX \times \sfY)$ is closed in $\cPg(\sfX \times \sfY)$ for the $W_p$ topology. This is clearly not the case if $\sfX = [0,1]$, or more generally if there exists at least one atomless measure supported on $\sfX$.
\end{remark}

Next we move to the characterization of relatively compact sets in $\rmT\rmL_p(\sfX;\sfY)$. We will rely on a notion of ``equidispersivity'' that we introduce, that one can think as an integral counterpart of equicontinuity.

\begin{definition}
A set $\cF \subseteq \rmT\rmL_p(\sfX;\sfY)$ is said equidispersed if: for every $\varepsilon > 0$, there exists $\delta > 0$ such that 
\begin{equation*}
\forall (\mu,f) \in \cF, \, \forall \eta \in \Gamma(\mu,\mu), \qquad \| \sfd_{\sfX}(x,x') \|_{L^p(\eta)} < \delta \quad \Rightarrow \quad \| \sfd_{\sfY}(f(x),f(x')) \|_{L^p(\eta)} < \varepsilon.
\end{equation*}
\end{definition}

Equivalently, it is standard to see that a set $\cF \subseteq \rmT\rmL_p(\sfX;\sfY)$ is equidispersed if there exists a modulus of continuity $\omega_\cF$ such that: % for every $(\mu,f) \in \cF$  
\begin{equation*}
\forall (\mu,f) \in \cF, \, \forall \eta \in \Gamma(\mu,\mu), \qquad	  \| \sfd_{\sfY}(f(x),f(x')) \|_{L^p(\eta)} \le \omega_\cF(\| \sfd_{\sfX}(x,x') \|_{L^p(\eta)}).
\end{equation*}
Here by modulus of continuity $\omega : \R_+ \to \R_+$ we mean a non-decreasing continuous function with $\omega(0) = 0$.

\begin{theorem}
\label{thm:rel_cpct_TL}
A subset $\cF \subseteq \rmT\rmL_p(\sfX;\sfY)$ is relatively compact if and only if $\iota(\cF)$ is relatively compact in $\cP_p(\sfX \times \sfY)$ and $\cF$ is equidispersed.
\end{theorem}

In particular, if $\sfd_\sfX$ and $\sfd_\sfY$ are bounded, the set $\cF \subseteq \rmT\rmL_p(\sfX;\sfY)$ is relatively compact if and only if $\iota(\cF)$ is tight and $\cF$ is equidispersed. In the case $\sfX$, $\sfY$ compact, the space $\cP(\sfX \times \sfY)$ is tight, thus $\cF$ is relatively compact if and only if it is equidispersed.

\begin{proof}
First assume that $\cF \subseteq \rmT\rmL_p(\sfX;\sfY)$ is relatively compact. Clearly it implies that $\iota(\cF)$ is relatively compact in $\cP_p(\sfX \times \sfY)$. Assume it is not equidispersed. That means that there exists $\varepsilon_0 > 0$, a sequence $(\delta_n)_{n \in \N}$ converging to $0$, and  $(\mu_n,f_n) \in \rmT\rmL_p(\sfX;\sfY)$, as well as $\eta_n \in \Gamma(\mu_n,\mu_n)$ such that for all $n$, $\| \sfd_{\sfY}(f_n(x),f_n(x')) \|_{L^p(\eta_n)} \ge \varepsilon_0 $ but $\| \sfd_{\sfX}(x,x') \|_{L^p(\eta_n)} < \delta_n$.
We write $\gamma_n = \iota(\mu_n,f_n)$, and $\sigma_n = (\pi^1 \times f_n \times \pi^2 \times f_n)_\sharp \eta_n$: we have $\sigma_n \in \Gamma(\gamma_n,\gamma_n)$, see Lemma~\ref{lm:couplings_sigma_eta}.
By relative compactness of $\cF$, up to extraction $\gamma_n$ converges in $\cP_p(\sfX \times \sfY)$ to a limit $\gamma \in \cPg_p(\sfX \times \sfY)$. By hemicontinuity and compactness of $\Gamma$, see Lemma~\ref{lm:Gamma_hemic}, up to extraction $(\sigma_n)_{n \in \N}$ converges to a limit $\sigma \in \Gamma(\gamma,\gamma)$.
We have
\begin{equation*}
\| \sfd_{\sfX}(x,x') \|_{L^p(\sigma)} = \lim_{n \to + \infty} \| \sfd_{\sfX}(x,x') \|_{L^p(\sigma_n)} = \lim_{n \to + \infty} \| \sfd_{\sfX}(x,x') \|_{L^p(\eta_n)} = 0,
\end{equation*}
while on the other hand 
\begin{equation*}
\| \sfd_{\sfY}(y,y') \|_{L^p(\sigma)} = \lim_{n \to + \infty} \| \sfd_{\sfY}(y,y') \|_{L^p(\sigma_n)} \geq \liminf_{n \to + \infty} \| \sfd_{\sfY}(f_n(x),f_n(x')) \|_{L^p(\eta_n)} \ge \varepsilon_0. 
\end{equation*}
That is, $\sigma \in \Gamma(\gamma,\gamma)$ with $\| \sfd_{\sfX}(x,x') \|_{L^p(\sigma)} = 0 $ but $\| \sfd_{\sfY}(y,y') \|_{L^p(\sigma)} > 0$: it contradicts $\gamma \in \cPg_p(\sfX \times \sfY)$. 

Conversely, assume $\cF$ is equidispersed and $\iota(\cF)$ is relatively compact in $\cP_p(\sfX \times \sfY)$, we call $\omega_\cF$ a modulus of equidispersivity. Let $(\mu_n,f_n)_{n \in \N}$ a sequence in $\cF$. As $\iota(\cF)$ is relatively compact, then $\gamma_n = \iota(\mu_n,f_n)$ converges as $n \to + \infty$, up to extraction, to $\gamma \in \cP_p(\sfX \times \sfY)$. We only need to prove that $\gamma \in \cPg_p(\sfX \times \sfY)$, we use the characterization of Lemma~\ref{lm:char_cpg}. So let $\sigma \in \Gamma(\gamma,\gamma)$. By hemicontinuity (see Lemma~\ref{lm:Gamma_hemic}), there exists a sequence $(\sigma_n)_{n \in \N}$ with $\sigma_n \in \Gamma(\gamma_n,\gamma_n)$ for all $n$ which converges to $\sigma$ in $\cP_p(\sfX \times \sfY \times \sfX \times \sfY)$. By Lemma~\ref{lm:couplings_sigma_eta}, we can write $\sigma_n = (\pi^1 \times f_n \times \pi^2 \times f_n)_\sharp \eta_n$ for some $\eta_n \in \Gamma(\mu_n,\mu_n)$.
In particular $\| \sfd_{\sfX}(x,x') \|_{L^p(\eta_n)}$ and $\| \sfd_{\sfY}(y,y') \|_{L^p(\sigma_n)} = \| \sfd_{\sfY}(f_n(x),f_n(x')) \|_{L^p(\eta_n)}$ converge respectively to $\| \sfd_{\sfX}(x,x') \|_{L^p(\sigma)}$ and $\| \sfd_{\sfY}(y,y') \|_{L^p(\sigma)}$ as $n \to + \infty$. By continuity of $\omega_\cF$,
\begin{align*}
\| \sfd_{\sfY}(y,y') \|_{L^p(\sigma)} & = \lim_{n \to + \infty} \| \sfd_{\sfY}(f_n(x),f_n(x')) \|_{L^p(\eta_n)} \\
& \le \lim_{n \to + \infty} \omega_\cF \left( \| \sfd_{\sfX}(x,x') \|_{L^p(\eta_n)} \right) = \omega_\cF \left( \| \sfd_{\sfX}(x,x') \|_{L^p(\sigma)} \right). 
\end{align*}
As $\omega_\cF(0) = 0$, we conclude to $\| \sfd_{\sfY}(y,y') \|_{L^p(\sigma)} = 0$ if $\| \sfd_{\sfX}(x,x') \|_{L^p(\sigma)} = 0$. It proves $\gamma \in \cPg_p(\sfX \times \sfY)$ thanks to Lemma~\ref{lm:char_cpg}. 
\end{proof}

With the help of Theorem~\ref{thm:rel_cpct_TL}, we recover the equivalent characterizations of convergence in $\rmT\rmL_p(\sfX\times \sfY)$ already studied in \cite{GarciaTrillos-Slepcev16}.
A sequence $(\eta_n)_{n \in \N} \in \cP_p(\sfX \times \sfX)$ is called \emph{stagnating} if 
\begin{equation*}
\lim_{n \to + \infty} \| \sfd_{\sfX}(x,x') \|_{L^p(\eta_n)} = 0. 
\end{equation*}
In particular if $\mu_n \to \mu$ in $\cP_p(\sfX)$ as $n \to + \infty$ and for all $n$, $\eta_n$ is an optimal transport plan between $\mu_n$ and $\mu$ then $(\mu_n)_{n \in \N}$ is stagnating. The following result has already been stated and proved in \cite[Proposition 3.12]{GarciaTrillos-Slepcev16}, we provide here another proof relying on equidispersivity. 

\begin{proposition}
\label{prop:char_conv_TLp}
Let $(\mu_n,f_n)_{n \in \N}$ a sequence in $\rmT\rmL_p(\sfX;\sfY)$ and $(\mu,f) \in \rmT\rmL_p(\sfX;\sfY)$. The followings are equivalent.
\begin{enumerate}
\item The sequence $(\mu_n,f_n)$ converges to $(\mu,f)$ in $\rmT\rmL_p(\sfX;\sfY)$ as $n\to + \infty$.
\item There exists a stagnating sequence $\eta_n \in \Gamma(\mu_n,\mu)$ such that 
\begin{equation*}
\lim_{n \to + \infty} \| \sfd_{\sfY}(f_n(x),f(x')) \|_{L^p(\eta_n)} = 0.
\end{equation*}
\item There holds $\mu_n \to \mu$ in $\cP_p(\sfX)$ and for every stagnating sequence $\eta_n \in \Gamma(\mu_n,\mu)$, 
\begin{equation*}
\lim_{n \to + \infty} \| \sfd_{\sfY}(f_n(x),f(x')) \|_{L^p(\eta_n)} = 0.
\end{equation*}
\end{enumerate}
\end{proposition}

\begin{proof}
\uline{$1 \Leftrightarrow 2$}. This equivalence is a direct consequence of~\eqref{eq:expr_TL_p} and the definition of a stagnating sequence. 

\uline{$2 \Rightarrow 3$}. Let $(\bar\eta_n)_{n \in \N}$ the stagnating sequence of Point 2, and let $(\eta_n)_{n \in \N}$ be another stagnating sequence. Then $W_p(\mu_n,\mu) \leq \| \sfd_{\sfX}(x,x') \|_{L^p(\bar \eta_n)}$ so $\mu_n \to \mu$ in $\cP_p(\sfX)$. Moreover by the gluing lemma \cite[Remark 5.3.3]{AGS} we can build $\sigma_n \in \Gamma(\mu_n,\mu,\mu)$ such that ${\pi^{1,2}}_\sharp \sigma_n = \bar\eta_n$ and ${\pi^{1,3}}_\sharp \sigma_n = \eta_n$. Then 
\begin{align*}
\| \sfd_{\sfY}(f_n(x),f(x')) \|_{L^p(\eta_n)} & = \| \sfd_{\sfY}(f_n(x_1),f(x_3)) \|_{L^p(\sigma_n)} \\
& \leq \| \sfd_{\sfY}(f_n(x_1),f(x_2)) \|_{L^p(\sigma_n)} + \| \sfd_{\sfY}(f(x_2),f(x_3)) \|_{L^p(\sigma_n)} \\
& = \| \sfd_{\sfY}(f_n(x),f(x')) \|_{L^p(\bar\eta_n)} + \| \sfd_{\sfY}(f(x_2),f(x_3)) \|_{L^p(\sigma_n)} 
\end{align*}
The first term goes to $0$ by assumption. As the singleton $\{ (\mu, f) \}$ is compact thus equidispersed by Theorem~\ref{thm:rel_cpct_TL}, it is enough to prove that $\| \sfd_{\sfX}(x_2,x_3) \|_{L^p(\sigma_n)}$ converges to $0$ to conclude that the second term converges to $0$. To that end we use the triangle inequality:
\begin{equation*}
\| \sfd_{\sfX}(x_2,x_3) \|_{L^p(\sigma_n)} \leq \| \sfd_{\sfX}(x_2,x_1) \|_{L^p(\sigma_n)} + \| \sfd_{\sfX}(x_1,x_3) \|_{L^p(\sigma_n)} = \| \sfd_{\sfX}(x,x') \|_{L^p(\bar \eta_n)} + \| \sfd_{\sfX}(x,x') \|_{L^p(\eta_n)},
\end{equation*}
and the right hand side converges to zero as both sequences are stagnating.

\uline{$3 \Rightarrow 2$}. It is immediate, because as $\mu_n \to \mu$ in $\cP_p(\sfX)$, there exists at least one stagnating sequence in $\Gamma(\mu_n,\mu)$: we take an optimal transport plan between $\mu_n$ and $\mu$. 
\end{proof}

\begin{remark}
With Proposition~\ref{prop:char_conv_TLp}, we recover more classical notions of convergence as particular cases of $\rmT\rmL_p(\sfX;\sfY)$ convergence: the latter can be interpreted as generalizing both $L^p$ and $\cP_p$ convergence. 

If $\mu \in \cP_p(\sfX)$ is fixed, then $(f_n)_{n \in \N}$ converges to $f \in L^p(\mu;\sfX;\sfY)$ if and only if $(\mu,f_n)$ converges to $(\mu,f)$ in $\rmT\rmL_p(\sfX;\sfY)$. It can be seen with Proposition~\ref{prop:char_conv_TLp}, taking $\eta_n = (\ii \times \ii)_\sharp \mu$, and was already proved in \cite[Lemma 5.4.1]{AGS}. 

On the other hand, if $f : \sfX \to \sfY$ is continuous and bounded, then $(\mu_n)_{n \in \N}$ converges to $\mu$ in $\cP_p(\sfX)$ if and only if $(\mu_n,f)$ converges to $(\mu,f)$ in $\rmT\rmL_p(\sfX;\sfY)$. The converse implication is immediate using Proposition~\ref{prop:char_conv_TLp}(3), and for the direct implication, if $\mu_n \to \mu$ in $\cP_p(\sfX)$, we use Proposition~\ref{prop:char_conv_TLp}(2) with $\eta_n$ an optimal transport plan between $\mu_n$ and $\mu$.
\end{remark}

\subsection{Graph convergence of continuous maps and convergence in $\rmT\rmL_p(\sfX;\sfY)$}

In this section we introduce the graph convergence of continuous maps with compact domains, and makes a link with $\rmT\rmL_p(\sfX;\sfY)$ convergence. Namely we prove a Lusin's style theorem for $\rmT\rmL_p(\sfX;\sfY)$ convergence, see Theorem~\ref{thm:Gconv_TLpconv} below. In addition to its intrinsic interest, it will be useful later in the proof of one of our core result (Corollary~\ref{crl:f_jointly_continuous}). Note how the criterion for compactness for graph convergence, which is a mere variation of Ascoli-Arzelà, bears some resemblance with equidispersivity for $\rmT\rmL_p(\sfX;\sfY)$.  

Let us quickly recall basic facts about the topology in the space 
	$\cK(\sfZ)$ of the compact 
	(nonempty) subsets of a metric space $(\sfZ,\sfd_\sfZ)$, we refer to \cite[Chapter 3]{aliprantis2006infinite} for standard results on this topology. 
$\cK(\sfZ)$ is a metric space with the Hausdorff metric:
\begin{equation*}
	\sfD_\sfZ(A,B):=\max\Big(
	\max_{a\in A}\sfd_\sfZ(a,B),
	\max_{b\in B}\sfd_\sfZ(b,A)\Big),
	\qquad
	\sfd_\sfZ(z,C):=\max_{c\in C}\sfd_Z(z,c).
\end{equation*}

A sequence $(K_n)_{n\in \N}$ converges to $K$ in $\sfD_\sfZ$ if and only if the following two conditions hold:
\begin{enumerate}
\item for every $z \in K$, we can find $(z_n)_{n \in \N}$ with $z_n \in K_n$ and $z_n$ converges to $z$;
\item for every sequence $(z_n)_{n \in \N}$ with $z_n \in K_n$, we can extract a subsequence converging to an element of $K$. 
\end{enumerate}

A collection 
$\cA\subset \cK(\sfZ)$ of compact sets  
is relatively compact if and only if 
the set $A:=\bigcup \cA$ is relatively compact
in $\sfZ$. 
In the case when $\cA$ is the image of a sequence
$(K_n)_{n\in \N}$, relative compactness is equivalent to say that 
every sequence $(z_n)_{n\in\N}$ with 
$z_n\in K_n$ for every $n\in \N$ has a convergent subsequence.

We can identify every map $f:\Dom f\to \sfY$  
with its graph
\begin{equation*}
\Graph f:=\Big\{(x,f(x)):x\in \Dom f\Big\}
\subset \sfX\times \sfY.   
\end{equation*}
If $\Dom f$ is compact, $f$ is continuous if and only if $\Graph{f}$ is compact as well,
so that the space of continuous maps 
with compact domain from $\sfX$ to $\sfY$ 
is isomorphic to the space of 
compact subsets $G$ of $\sfZ=\sfX\times \sfY$ 
which are concentrated on a graph, that is, which satisfy~\eqref{eq:char_graph}.
We say that $f$ is a \tCC-map, i.e.~a map with compact graph.

\begin{definition}%[\tCC-maps and graph convergence]
We denote by $\CCmaps \sfX\sfY$ 
the set of continuous maps $f:\Dom f\to \sfY$
with nonempty compact domain $\Dom f\subset \sfX$ 
(or equivalently, the set of maps with compact graph).

The topology in $\CCmaps\sfX\sfY$ is induced by the embedding of 
this space in $\cK(\sfX\times \sfY)$ given by
$f\mapsto \Graph{f}$. 
In particular, a sequence $f_n:\Dom{f_n}\to \sfY$, $\Dom{f_n}\in \cK(\sfX)$
of \tCC-maps 
is \emph{G}-converging to $f:\Dom f\to \sfY$
if $\Graph{f_n}$ is converging 
to $\Graph{f}$ in $\cK(\sfX\times \sfY)$.
\end{definition}

It is not difficult to check that 
$\rm{G}$-convergence is equivalent to the following two conditions:
\begin{enumerate}
\item $\Dom{f_n}\to\Dom f$ 
in $\cK(\sfX)$;
\item for every increasing sequence 
$k\mapsto n_k$ and
every sequence $x_{n_k}\in \Dom{f_{n_k}}$
converging to $x\in \Dom f$ we have
$f_{n_k}(x_{n_k})\to f(x)$
as $k\to +\infty$.
\end{enumerate}

We prove the analogue of Theorem~\ref{thm:rel_cpct_TL} and characterize relative compactness in $\CCmaps \sfX\sfY$. In this case, it is simpler and takes the form of a generalization of Ascoli-Arzelà. 

\begin{definition}
A subset $\cF\subseteq \CCmaps\sfX\sfY$ is called equicontinuous if: for every $\varepsilon > 0$, there exists $\delta > 0$ such that
\begin{equation*}
\forall f \in \cF, \quad \forall x,x' \in \Dom{f}, \qquad
\sfd_\sfX(x,x') < \delta \quad \Rightarrow \quad \sfd_\sfY(f(x),f(x')) < \varepsilon.
\end{equation*}
\end{definition}

This is equivalent to have a modulus of continuity $\omega_\cF$ such that, for $f \in \cF$, for $x,x' \in  \Dom{f}$ we have $\sfd_\sfY(f(x),f(x')) \le \omega_\cF(\sfd_\sfX(x,x'))$.

\begin{theorem}
A set $\cF\subset \CCmaps\sfX\sfY$ is relatively compact (with respect to {\rm G}-convergence in $\CCmaps\sfX\sfY$) if and only if $\bigcup_{f \in \cF} \Graph f$ is relatively compact in $\sfX \times \sfY$ and $\cF$ is equicontinuous. 
\end{theorem}

\begin{proof}
Assume $\cF$ is relatively compact in $\CCmaps\sfX\sfY$. Then $\{ \Graph f : f \in \cF \}$ is relatively compact in $\cK(\sfX \times \sfY)$, thus $\bigcup_{f \in \cF} \Graph f$ is relatively compact. Assume by contradiction that $\cF$ is not equicontinuous. There exists $\varepsilon_0 > 0$, a sequence $(\delta_n)_{n \in \N}$ converging to $0$, and $f_n$ in $\cF$, $x_n, x'_n \in \Dom{f_n}$ such that $\sfd_\sfX(x_n,x'_n) < \delta_n$ but $\sfd_\sfY(f_n(x_n),f_n(x'_n)) \ge \varepsilon_0$. Up to extraction $(f_n)_{n \in \N}$ $\rmG$-converges to a limit $f$, and also $x_n$ and $x'_n$ converge to limits $x$ and $x'$ as $n \to + \infty$. By $\rmG$-convergence, we necessarily have $x,x' \in \Dom f$ and $f_n(x_n)$, $f_n(x'_n)$ converge to $f(x)$ and $f(x')$. Passing to the limit we have $\sfd_\sfX(x,x') = 0$, but $\sfd_\sfY(f(x),f(x')) \geq \varepsilon_0 > 0$. This is a contradiction. 

Conversely, assume $\bigcup_{f \in \cF} \Graph f$ is relatively compact in $\sfX \times \sfY$ and $\cF$ is equicontinuous. Let $(f_n)_{n \in \N}$ a sequence in $\cF$. As $\bigcup_{n \in \N} \Graph{f_n}$ is relatively compact in $\cK(\sfX \times \sfY)$, we can extract from it a sequence converging to $G \in \cK(\sfX \times \sfY)$. We only need to show that $G$ is the graph of a function. If $(x,y)$ and $(x',y')$ are two points in $G$, then we can find sequences $(x_n,f_n(x_n))$ and $(x'_n,f_n(x'_n))$ in $\Graph{f_n}$ converging to $(x,y)$ and $(x',y')$ as $n \to + \infty$. Then
\begin{equation*}
\sfd_\sfY(y,y')  = \lim_{n \to + \infty} \sfd_\sfY(f(x_n),f(x'_n))  \le \lim_{n \to + \infty} \omega_\cF(\sfd_\sfX(x_n,x'_n))  
  = \omega_{\cF}(\sfd_\sfX(x,x')).
\end{equation*}
In particular if $x=x'$ then $y=y'$, thus $G$ is the graph of a function. It concludes the proof.
\end{proof}

Eventually we move the main result of this section: we want to relate convergence in $\rmT\rmL_p(\sfX; \sfY)$ to graph convergence. We prove an analogue of Lusin's theorem: if there is convergence in $\rmT\rmL_p(\sfX; \sfY)$, for any $\varepsilon > 0$, up to ignoring sets of measure less than $\varepsilon$, the convergence in $\rmT\rmL_p(\sfX; \sfY)$ can be upgraded to $\rm{G}$-convergence. In the sequel we write $f \restr{K}$ for the restriction of the function $f$ over a subset $K$.

\begin{theorem}
\label{thm:Gconv_TLpconv}
Assume that $\sfd_\sfX$ and $\sfd_\sfY$ are bounded and let $(\mu_n,f_n)_{n \in \N}$ and $(\mu,f)$ elements of $\rmT\rmL_p(\sfX;\sfY)$. 

Then $(\mu_n,f_n)_{n \in \N}$ converges to $(\mu,f)$ in $\rmT\rmL_p(\sfX;\sfY)$ if and only if: for every $\varepsilon > 0$, there exist measures $\mu_{\eps,n}$ and $\mu_\eps$ with compact support $K_{\varepsilon,n}$ and $K_\varepsilon$ such that
\begin{equation*}
\mu_{\eps,n}\le \mu_n,\quad
\mu_\eps\le \mu,\quad
\mu_{\eps,n}(\sfX)\ge 1-\eps,\quad 
\mu_{\eps,n}\to \mu_\eps 
\end{equation*}
and ${f_n} \restr{K_{\varepsilon,n}}$ is continuous and $\rm{G}$-converges to the continuous function $f\restr{K_\varepsilon}$ as $n \to + \infty$.  
\end{theorem}

\begin{proof}
We write $\gamma_n = \iota(\mu_n,f_n) = (\ii \times f_n)_\sharp \mu_n$ and $\gamma = \iota(\mu,f) = (\ii \times f)_\sharp \mu$. We also take $(\Omega,\frB,\P)$ a standard Borel space. 

\uline{Direct implication}. Convergence of $(\mu_n,f_n)$ to $(\mu,f)$ in $\rmT\rmL_p(\sfX;\sfY)$ means that $\gamma_n$ converges to $\gamma$ in $\cP_p(\sfX \times \sfY)$, thus by the Skorokhod embedding theorem (Theorem~\ref{thm:skorokhod_Wp}) we can find $(X_n,Y_n)$ random variables with $(X_n,Y_n)_\sharp \P = \gamma_n$ which converge to $(X,Y) \in L^p(\Omega;\sfX \times \sfY)$ with $(X,Y)_\sharp \P = \gamma$. In particular, $Y_n = f_n(X_n)$ and $Y= f(X)$ almost surely.

Fix $\varepsilon > 0$. By a combination of Lusin and Egoroff's theorem, we can find a compact $S_\eps \subset \Omega$ with $\P(S_\varepsilon) \geq 1 - \varepsilon$ such that $X_n, Y_n, X,Y$ are continuous over $S_\varepsilon$ and the convergence of $(X_n,Y_n)$ to $(X,Y)$ is uniform over $S_\eps$. We define $\mu_{\varepsilon,n} = {X_n}_\sharp( \P \mres S_\eps)$ and $\mu_\varepsilon = X_\sharp (\P \mres S_\eps)$, and naturally $K_{\eps,n} = X_n(S_\eps)$ and $K_\eps = X(S_\eps)$. It gives directly $\mu_{\eps,n}\le \mu_n$, $\mu_{\eps}\le \mu$, and the convergence $\mu_{\eps,n}\to \mu_\eps$. 

Moreover, the graph of $f_n \restr{K_{\eps,n}}$ is nothing else than $(X_n,Y_n)(S_\eps)$, which is compact, while the graph of $f \restr{K_\eps}$ is the compact set $(X,Y)(S_\eps)$. As $(X_n,Y_n)$ converges uniformly to $(X,Y)$ over $S_\eps$, it is easy to check that $(X_n,Y_n)(S_\eps)$ converges in $\cK(\sfX \times \sfY)$ to $(X,Y)(S_\eps)$.  

\uline{Converse implication}. Let $\phi : \sfX \times \sfY \to \R$ a continuous and bounded test function, we call $M$ an upper bound of $|\phi|$. For $\eps > 0$ we have 
\begin{equation*}
\int_{\sfX \times \sfY} \phi(x,y) \, \d \gamma_n(x,y) = \int_{\sfX } \phi(x,f(x)) \, \d \mu_n(x) \leq M \eps + \int_{\sfX} \phi(x,f_n(x)) \, \d \mu_{\eps,n}(x).
\end{equation*}
Again by the Skorokhod embedding theorem (Theorem~\ref{thm:skorokhod_Wp}) there exists a sequence $(X_{\eps,n})_{n \in \N}$ valued in $L^p(\Omega;\sfX)$ with ${X_{\eps,n}}_\sharp ( \mu_{\eps,n}(\sfX) \P) = \mu_{\eps,n}$ which converges to $X_\eps \in L^p(\Omega;\sfX)$ with ${X_{\eps}}_\sharp (\mu_{\eps}(\sfX) \P) = \mu_{\eps}$. As $X_{\eps,n}(\omega) \in K_{\eps,n}$ for all $n$ for a.e. $\omega$, and thanks to continuity of $\phi$ and $\rmG$-convergence, $\phi(X_{\eps,n}(\omega),f_n(X_{\eps,n}(\omega)))$ converges to $\phi(X_\eps(\omega),f(X_\eps(\omega)))$ for a.e. $\omega$. Thus by dominated convergence
\begin{align*}
\lim_{n\to + \infty} \int_{\sfX} \phi(x,f_n(x)) \, \d \mu_{\eps,n}(x)  & = \lim_{n \to + \infty} \mu_{\eps,n}(\sfX) \int_{\Omega} \phi(X_{\eps,n}(\omega),f_n(X_{\eps,n}(\omega))) \, \d \P(\omega) \\
& = \mu_{\eps}(\sfX)\int_{\Omega} \phi(X_{\eps}(\omega),f(X_{\eps}(\omega))) \, \d \P(\omega) \\
& = \int_{\sfX} \phi(x,f(x)) \, \d \mu_{\eps}(x) \leq M \eps +  \int_{\sfX} \phi(x,f(x)) \, \d \mu(x). 
\end{align*}
Putting pieces together we obtain 
\begin{equation*}
\limsup_{n \to + \infty} \int_{\sfX \times \sfY} \phi(x,y) \, \d \gamma_n(x,y) \leq 2M \eps + \int_{\sfX} \phi(x,f(x)) \, \d \mu(x). 
\end{equation*}
Taking the limit $\eps \to 0$, and then permuting the role of $\phi$ and $-\phi$, we conclude that $\int \phi \d \gamma_n$ converges to $\int \phi \d \gamma$. The conclusion follows as $\phi$ is an arbitrary bounded continuous function.
\end{proof}

\section{Generic discrete measures}
\label{sec:generic_measures}

We turn to the second set of preliminary results for our study: we define the class of generic measures over which the transport representative can be identified unequivocally. 

For an integer $m$ we write $\Delta_m$ for the $(m-1)$ dimensional simplex, that is $\Delta_m$ is the set of $\aa = (a_i)_{1 \leq i \leq m}$ such that $a_i \geq 0$ for all $i =1,\ldots,m$ and $\sum_{i=1}^m a_i = 1$. We call $\mu$ an empirical distribution if $\mu = \frac{1}{m} \sum_{i=1}^{m}  \delta_{x_i}$ for some $(x_i)_{1 \le i \le m}$ not necessarily distinct, equivalently if $\mu = \sum_{i=1}^{m'} a_i \delta_{x'_i}$ where $(a_i)_{1 \le i \le m'} \in \Delta_{m'}$ and all $a_i$ are rational numbers.

\begin{definition}
\label{def:generic_weights}
An element $\aa \in \Delta_m$ is called generic if, for any $I_1, I_2$ subsets of $\{1, \ldots,m \}$,
\begin{equation*}
\sum_{i \in I_1} a_i = \sum_{i \in I_2} a_i \quad \Rightarrow \quad I_1 = I_2. 
\end{equation*}
A probability measure $\mu \in \cP(\sfX)$ is called generic if $\mu = \sum_{i=1}^m a_i \delta_{x_i}$ for $\aa \in \Delta_m$ generic and $(x_i)_{1 \leq i \leq m} \in \sfX^m$.
\end{definition}

Necessarily if $\aa$ is generic then $a_i > 0$ for all $i$. The uniform distribution $(1/m, \ldots, 1/m)$ is not generic for $m \geq 2$: it is quite the opposite as necessarily if $\aa$ is generic then all weights $a_i$ are distinct. 

% For instance, the uniform distribution $(1/m, \ldots, 1/m)$ over $\Delta_m$ is not generic while $(1/7,2/7,4/7) \in \Delta_3$ is.

\begin{proposition}
\label{prop:super_generic_dense}
For any $m \geq 1$,
the set of generic weights is open and dense in $\Delta_m$. In particular, it is dense in $\Delta_m \cap \Q^m_+$.
\end{proposition}

\begin{proof}
The set of generic weights is the intersection of $ \left\{ \aa \in \Delta_m \ : \ \sum_{i \in I_1} a_i \neq \sum_{i \in I_2} a_i \right\}$ where $I_1, I_2$ are taken to be any pair of distinct subsets of $\{1, \ldots, m \}$, thus is open as a finite intersection of open sets. 

For denseness, if $\bb$ is any collection of $m$ non-negative elements which is $\Q$-linearly independent, then $(b_1/s, \ldots, b_m /s)$ with $s = \sum_{i=1}^m b_i$ belongs to $\Delta_m$ and is generic. The conclusion follows as the set of $\Q$-linearly independent collection of $m$ numbers is dense in $\R_+^m$.
\end{proof}

\begin{corollary}
\label{crl:dense_supergeneric}
The set of empirical generic measures is dense in $\cP_p(\sfX)$.
\end{corollary}

\begin{proof}
The set of empirical measures is dense in $\cP_p(\sfX)$, see e.g.~\cite[Theorem 6.18]{villani2009optimal}.  
By Proposition~\ref{prop:super_generic_dense}, any empirical measure can be approximated by a generic measure with weights which are rational. The conclusion follows.
\end{proof}

The next result explains the main interest of generic measures: the location of the atoms can be identified from the $\mu$, as the weights $\aa$ encode the ``ordering'' of the atoms. 

\begin{proposition}
\label{prop:super_generic_injective}
Fix $\aa \in \Delta_m$ generic. Then the map 
\begin{equation*}
\frG_{\aa} : (x_1, \ldots, x_m) \mapsto \sum_{i=1}^m a_i \delta_{x_i}    
\end{equation*}
is continuous and injective from $\sfX^m$ into $\cP(\sfX)$. Its inverse $\frG_\aa^{-1}$, defined over $\frG_\aa(\sfX^m) \subset \cP_p(\sfX)$ and valued into $\sfX^m$, is continuous. 
\end{proposition}

We leave to the reader to check that, if the set $\sfX$ is infinite, injectivity of $\frG_\aa$ is a characterization of the genericity of $\aa$. 

\begin{proof}
Continuity of $\frG_\aa$ is clear, see also~\eqref{eq:estimate_trivial} below. Writing $\xx = (x_1, \ldots, x_m)$ and $\xx' = (x'_1, \ldots, x'_m)$, assume $\mu = \frG_\aa(\xx) = \frG_\aa(\xx')$ and let $x \in \sfX$. Let $I$ (resp. $I'$) the set of $i$ with $x_i = x$ (resp. $x'_i = x$). Then $\sum_{i \in I} a_i = \sum_{i \in I'} a_i = \mu(\{ x \})$ and thus $I=I'$. As $x$ is arbitrary, we get $\xx = \xx'$.

Moreover $\frG_\aa$ is proper, meaning that the inverse image of a compact set is compact: this can easily be seen from tightness, which characterize relatively compact sets of $\cP(\sfX)$. Thus $\frG_\aa^{-1}$ is continuous on $\frG_\aa(\sfX^m)$, as inverse of a continuous injective proper map. 
\end{proof}

\begin{corollary}
\label{crl:atmost_one_f_sg}
Fix $\mu \in \cP(\sfX)$ generic and $f, f' : \sfX \to \sfY$. Then if $f_\sharp \mu = f'_\sharp \mu$ we have $f = f'$ on $\supp(\mu)$. 
\end{corollary}

\begin{proof}
Write $\mu = \sum_{i=1}^m a_i \delta_{x_i}$, that is, $\mu = \frG_\aa(x_1, \ldots, x_m)$ with $\aa \in \Delta_m$ generic. Then $f_\sharp \mu = \frG_\aa(f(x_1), \ldots, f(x_m))$. If $f_\sharp \mu = f'_\sharp \mu$ as $\frG_\aa$ is injective (Proposition~\ref{prop:super_generic_injective}) we have $(f(x_1), \ldots, f(x_m)) = (f'(x_1), \ldots, f'(x_m))$, which exactly means $f = f'$ on $\supp(\mu)$.   
\end{proof}

We also prove a quantitative strengthening of Proposition \ref{prop:super_generic_injective} which will be key to prove our main technical estimate, namely Proposition~\ref{prop:estimate_discrete}. We endow $\sfX^m$ with the $\ell_p$ distance weighted by $\aa$: 
\begin{equation*}
\sfd_{\aa,\sfX^m}(\xx,\xx') = \left( \sum_{i=1}^m a_i \sfd_{\sfX}(x_i,x'_i)^p \right)^{1/p}.   
\end{equation*}
The trivial coupling always yields, if $\mu = \frG_\aa(\xx)$ and $\mu' = \frG_\aa(\xx')$, that
\begin{equation}
\label{eq:estimate_trivial}
W_p(\mu,\mu') \leq \sfd_{\aa,\sfX^m}(\xx,\xx').  
\end{equation}
We prove a refined estimate about the equality case.

\begin{proposition}
\label{proposition_local_isometry}
Fix $\aa \in \Delta_m$ generic. For $\mu = \frG_\aa(\xx)$ with $\xx \in \sfX^m$, there exists $\varepsilon > 0$ such that, for any $\mu' = \frG_\aa(\xx')$, 
\begin{equation}
\label{eq:local_isom}
\min( \sfd_{\aa,\sfX^m}(\xx,\xx'), W_p(\mu,\mu') ) \leq \varepsilon \quad  \Rightarrow \quad W_p(\mu,\mu') = \sfd_{\aa,\sfX^m}(\xx,\xx').
\end{equation}
\end{proposition}

It does not imply that $\frG_\aa$ is a local isometry, as only the distances from $\mu = \frG_\aa(\xx)$ to any point $\mu' = \frG_\aa(\xx')$ in a ball centered at $\mu$ are preserved, not all distances between points of the ball. 

\begin{proof}
We write $\mu = \sum_{i=1}^m a_i \delta_{x_i}$ with the $x_i$ not necessarily distinct. We write $\theta(\mu)$ for the minimal distance between distinct elements of $\{ x_1, \ldots, x_m \}$. We also write $r(\aa)$ for the minimal value of the elements of $\aa$. The key estimate is:
\begin{equation}
\label{eq:local_isometry_partial}
\sfd_{\aa,\sfX^m}(\xx,\xx') \leq \frac{\theta(\mu) r(\aa)^{1/p}}{2} \quad \Rightarrow \quad W_p(\mu,\mu') = \sfd_{\aa,\sfX^m}(\xx,\xx'). 
\end{equation}
For that we only need to show that the coupling sending $x_i$ onto $x'_i$ for all $i$ is optimal in this case. It is enough to prove that the coupling is $c$-cyclically monotone \cite[Theorem 6.1.4]{AGS}. If $i_1, \ldots, i_K$ enumerate a subset of $\{ 1, \ldots, m \}$ with $K \leq m$ elements such that $x_{i_1}, \ldots, x_{i_K}$ are all distinct, identifying $K+1$ with $1$, we need to check that 
\begin{equation*}
\sum_{k=1}^K \sfd_{\sfX}(x_{i_k}, x'_{i_k})^p \leq \sum_{k=1}^K \sfd_{\sfX}(x_{i_k}, x'_{i_{k+1}})^p.  
\end{equation*}
From $\sfd_{\aa,\sfX^m}(\xx,\xx') \leq \theta(\mu) r(\aa)^{1/p} /2$ we have that $\sfd_{\sfX}(x_{i_k}, x'_{i_k}) \leq \theta(\mu)/2$ for all $k$. On the other hand by the triangle inequality we necessarily have $\sfd_{\sfX}(x_{i_k}, x'_{i_{k+1}}) \geq \sfd_{\sfX}(x_{i_k}, x_{i_{k+1}}) - \sfd_{\sfX}(x'_{i_{k+1}},x_{i_{k+1}} ) \geq \theta(\mu)/2$ for all $k$. The $c$-cyclical monotonicity follows, and thus~\eqref{eq:local_isometry_partial}. 

On the other hand, as $\frG_\aa$ has continuous inverse (Proposition~\ref{prop:super_generic_injective}), we can find $\delta > 0$ such that if $\mu' = \frG_\aa(\xx') \in \frG_\aa(\sfX^m)$ with $W_p(\mu,\mu') < \delta$ then $\sfd_{\aa,\sfX^m}(\xx,\xx') < \theta(\mu) r(\aa)^{1/p}/2$. 

The conclusion~\eqref{eq:local_isom} follows from~\eqref{eq:local_isometry_partial} by taking $\varepsilon = \min(\theta(\mu) r(\aa)^{1/p}/2, \delta)$.
\end{proof}

\section{Existence, measurability and continuity of transport representatives}
\label{sec:main_results}

We now have all the tools to state and prove our core results about existence, measurability, and continuity of transport representatives. We start by rewriting what we mean by ``transport representative'' with the notations of Section~\ref{sec:TLp}.

\begin{definition}
\label{def:meas_continuous_lipschitz_TLp}
A map $F : \cP_p(\sfX) \to \cP_p(\sfY)$ has a \emph{transport representative} if there exists $\frF : \cP_p(\sfX) \to \rmT\rmL_p(\sfX;\sfY)$ such that $\frF(\mu) = (\mu,\ff(\cdot,\mu))$ and $\ff(\cdot,\mu)_\sharp \mu = F(\mu)$ for all $\mu \in \cP_p(\sfX)$.

Being $\cP_p(\sfX)$, $\rmT\rmL_p(\sfX;\sfY)$ endowed with their canonical topology and their Borel $\sigma$-algebra, the representative is said Borel (resp. continuous) if $\frF$ is Borel (resp. continuous).   
\end{definition}

Having a transport representative means that for every $\mu \in \cP_p(\sfX)$ there exists $f : \sfX \to \sfY$ Borel such that $f_\sharp \mu = F(\mu)$. Indeed in this case necessarily $f \in L^p(\mu;\sfX;\sfY)$ so that $(\mu,f) \in \rmT\rmL_p(\sfX;\sfY)$ and we define $\frF(\mu) = (\mu,f)$. The structure of the space $\rmT\rmL_p(\sfX;\sfY)$ plays no role to analyze the existence of the transport representative, only its Borel measurability or continuity.

\begin{remark}
\label{rmk:mes_cont}
If $\frF$ is a Borel transport representative, then we can find $\ff : \sfX \times \cP_p(\sfX) \to \sfY$ \emph{jointly Borel} in $(x,\mu)$ such that $\frF(\mu) = (\mu,\ff(\cdot,\mu))$ and $\ff(\cdot,\mu)_\sharp \mu = F(\mu)$ for all $\mu \in \cP_p(\sfX)$. Indeed with Proposition~\ref{prop:represent_TLp} and the notations therein we simply define $\ff(x,\mu) = \frS(x,\frF(\mu))$.
Thus Borel measurability in the sense of Definition~\ref{def:meas_continuous_lipschitz_TLp} coincides with (joint) Borel measurability of a $\ff : \sfX \times \cP_p(\sfX) \to \sfY$ satisfying~\eqref{eq:intro_transport_rep}.

On the other hand if $\frF$ is a continuous transport representative in the sense of Definition~\ref{def:meas_continuous_lipschitz_TLp} it only corresponds to continuity with respect to the $\mu$ variable (of the map $\frF$), but it does not imply anything on continuity of $\ff(x,\mu)$ in the variable $x$.
\end{remark}

Recalling that $\iota(\mu,f) = (\ii \times f)_\sharp \mu$ is the isometric bijection between $\rmT\rmL_p(\sfX; \sfY)$ and $\cPg_p(\sfX \times \sfY)$, a mapping $\frF$ from $\cP_p(\sfX)$ to $\rmT\rmL_p(\sfX; \sfY)$ is a transport representative of $F$ if and only if 
\begin{equation}
\label{eq:transportlift_char}
\iota(\frF(\mu)) \in \Gamma(\mu,F(\mu)) \qquad \text{for all } \mu \in \cP_p(\sfX).
\end{equation}
In particular if $F$ admits a transport representative $\frF$, then it can be reconstructed easily from $\frF$ via the formula
\begin{equation}
\label{eq:FfromTlift}
F(\mu) = \pi^2_\sharp \left( (\ii \times \ff(\cdot,\mu))_\sharp \mu \right) = \pi^2_\sharp  (\iota \circ \frF)(\mu)  \qquad \text{for all } \mu \in \cP_p(\sfX).
\end{equation}

\subsection{Borel measurable transformation $F$}

In case we are only interested in Borel measurability, then without too much surprise Borel measurability of $\frF$ is a consequence of Borel measurability of the map $F$.

\begin{theorem}
\label{thm:measurable}
Assume $F : \cP_p(\sfX) \to \cP_p(\sfY)$ has at least one transport representative $\frF$. Then it has a Borel transport representative if and only if $F$ is Borel.
\end{theorem}

\begin{proof}
Assume first that $F$ has a Borel transport representative. From~\eqref{eq:FfromTlift}, expressing $F$ in terms of its transport representative, we see that $F$ is Borel.

Conversely, assume that $F$ is Borel. Let $G : \cP_p(\sfX) \twoheadrightarrow \cPg_p(\sfX \times \sfY)$ the correspondence defined by:
\begin{equation*}
G(\mu) = \Gamma(\mu, F(\mu)) \cap \cPg_p(\sfX \times \sfY).
\end{equation*}
Given~\eqref{eq:transportlift_char}, it is enough to prove that $G$ has a least a Borel selection, as in this case we compose any such selection with $\iota^{-1}$ to obtain a Borel transport representative. As $(\cPg_p(\sfX \times \sfY),W_p)$ is isometric to $\rmT\rmL_p(\sfX;\sfY)$ thus metric, separable and completely metrizable (see Theorem~\ref{thm:tl_polish}), by the Kuratowski–Ryll-Nardzewski selection Theorem \cite[Theorem 18.13]{aliprantis2006infinite} we only need to prove that $G$ has closed non-empty values and is weakly measurable. The selection $G$ has closed values as $\Gamma(\mu,F(\mu))$ is closed, and is never empty as there exists a transport representative by assumption. Thus we need to show that $G$ is weakly measurable, in the sense that lower inverse of open sets are Borel \cite[Definition 8.1]{aliprantis2006infinite}. So let $B \subseteq \cPg_p(\sfX \times \sfY)$ an open set.
The lower inverse image of $B$ is  
\begin{equation}
\label{eq:aux_proof_meas}
\{ \mu \in \cP_p(\sfX) \ : \ G(\mu) \cap B \neq \emptyset \} = \{ \mu \in \cP_p(\sfX) \ : \ (\mu,F(\mu)) \in C  \},
\end{equation}
where we define $C = \{ (\mu, \nu) \in \cP_p(\sfX) \times \cP_p(\sfY) \ : \ \Gamma(\mu,\nu) \cap B \neq \emptyset  \}$.
As $F$ is Borel, it is enough to prove that $C$ is a Borel subset of $\cP_p(\sfX) \times \cP_p(\sfY)$. With $\tilde{B}$ an open set of $\cP_p(\sfX \times \sfY)$ such that $B = \tilde{B} \cap \cPg_p(\sfX \times \sfY)$ and $H_m$ the closed set of $\cP_p(\sfX \times \sfY)$ defined in~\eqref{eq:def_Hm} we have 
\begin{equation*}
C = \bigcap_{m \geq 1} \{ (\mu, \nu) \ : \ \Gamma(\mu,\nu) \cap \tilde{B} \cap H_m^c \neq \emptyset  \}.
\end{equation*}
Each of the set in the intersection in the right hand side is Borel by hemicontinuity of $\Gamma$ (as lower inverse image of an open set), thus $C$ is Borel as countable intersection of Borel sets. The weak Borel measurability of $G$ follows from the Borel measurability of $F$ and~\eqref{eq:aux_proof_meas}.
\end{proof}

\subsection{Continuous transformation $F$}

We now assume that $F$ is continuous. The main message is that in this case we cannot a priori guarantee that a transport representative, if it exists, is continuous. 

Let us first focus on the existence of the representative. Any obstruction to the existence, if it exists, must manifest at the level of the empirical measures. We first recall a well-known result when the source measure $\mu$ is atomless.

\begin{lemma}
\label{lm:existence_map_atomless}
If $\mu \in \cP(\sfX)$ has no atoms, then for any $\nu \in \cP(\sfY)$ there exists $f$ Borel with $f_\sharp \mu = \nu$.
\end{lemma}

\begin{proof}
As $\mu$ has no atoms we can find a Borel map from $\sfX$ to $[0,1]$ sending $\mu$ onto the Lebesgue measure,  \cite[Proposition 9.1.11]{bogachev2007measure}. Moreover, using a Borel isomorphism between a subset of $[0,1]$ and $\sfY$ \cite[Theorem 1.1]{kallenberg2017random}, we can assume $\sfY \subseteq [0,1]$. The result follows as any probability on $[0,1]$ can be represented as the push-forward of the Lebesgue measure by a Borel map, e.g. by the quantile function.
\end{proof}

Building on this result, we can conclude to a sufficient condition for the existence of a transport representative.

\begin{proposition}
\label{prop:existence_non_splitting}
If $F : \cP_p(\sfX) \to \cP_p(\sfY)$ is continuous and non-splitting on empirical measures, then there exists a transport representative, meaning for every $\mu \in \cP_p(\sfX)$ there exists $f : \sfX \to \sfY$ Borel such that $f_\sharp \mu = F(\mu)$. 
\end{proposition}

\begin{proof}
Let us consider $\mu \in \cP(\sfX)$, we decompose it into its atomic and diffuse part:
\begin{equation*}
\mu = \sum_{m = 1}^M a_m \delta_{x_m} + \mu_\d,
\end{equation*}
where $a_m > 0$ for all $m$, $M \in \N \cup \{ + \infty \}$, the $x_m$ are all distinct, and $\mu_\d \in \cM(\sfX)$ a diffuse measure, i.e. without atoms. Consider $(\aa^{(n)})_n$ a sequence valued in $\Q_+^M$ such that $\aa^{(n)}$ has finite support for any $n$, and such that, for every $m \geq 1$, $a^{(n)}_m$ is non-decreasing and converges to $a_m$ as $n \to + \infty$. We can also find $\mu_\d^{(n)}$ a sequence of empirical measures converging in $\cP_p(\sfX)$ to $\mu_\d$, see \cite[Theorem 6.18]{villani2009optimal}. We write 
\begin{equation*}
\mu^{(n)} = \sum_{m = 1}^M a^{(n)}_m \delta_{x_m} + \mu_\d^{(n)},
\end{equation*}
that we normalize in such a way that $\mu^{(n)}$ is a probability measure. 
It defines an empirical measure and $\mu^{(n)}$ converges to $\mu$ as $n \to + \infty$. By continuity $F(\mu^{(n)})$ converges to $F(\mu)$ as $n \to + \infty$. Moreover as $F$ is non-splitting on empirical measures, we can find $f^{(n)} :\sfX \to \sfY$ such that $F(\mu^{(n)}) = f^{(n)}_\sharp \mu^{(n)}$: we write it as
\begin{equation}
\label{eq:aux_proof_existence}
F(\mu^{(n)}) = f^{(n)}_\sharp \left( \sum_{m = 1}^M a^{(n)}_m \delta_{x_m} + \mu^{(n)}_\d \right) = \sum_{m = 1}^M a^{(n)}_m \delta_{f^{(n)}(x_m)} + f^{(n)}_\sharp \mu_\d^{(n)}.  
\end{equation}
For a fixed $m \geq 1$, by tightness of $\{ F(\mu^{(n)}) : n \ge 1 \}$ and as $a^{(n)}_m \geq a_m/2 > 0$ for $n$ large enough, then $(f^{(n)}(x_m))_{n \in \N}$ belongs to a compact set, and thus up to extraction converges to a limit $y_m$ as $n \to + \infty$. By a diagonal argument, up to extraction the convergence $f^{(n)}(x_m) \to y_m$ holds for all $m \geq 1$. Thus the first summand in the right hand side of~\eqref{eq:aux_proof_existence} converges narrowly to $\sum_{m} a_m \delta_{y_m}$. For the second term in~\eqref{eq:aux_proof_existence}, it is tight as the left hand side is, thus up to a further extraction $f^{(n)}_\sharp \mu^{(n)}_\d$ converges to $\nu \in \cM(\sfY)$. We deduce, passing to the limit in~\eqref{eq:aux_proof_existence}
\begin{equation*}
F(\mu) = \sum_{m = 1}^M a_m \delta_{y_m} + \nu.
\end{equation*}
Eventually we can construct $f$ such that $f_\sharp  \mu = F(\mu)$. Indeed by Lemma~\ref{lm:existence_map_atomless} there exists $f'$ such that $f'_\sharp \mu_\d = \nu$ and define $f(x_m) = y_m$ for any $m$, and $f(x) = f'(x)$ if $x$ is not an atom of $\mu$. 
\end{proof}

We now have all the tools to prove the first main result announced in the introduction: that $F$ continuous and non-splitting implies the existence of a measurable transport representative.

\begin{proof}[\textbf{Proof of Theorem~\ref{thm:main_continuity}}]
We assume that $F : \cP_p(\sfX) \to \cP_p(\sfY)$ is continuous and non-splitting on empirical measures. By Proposition~\ref{prop:existence_non_splitting} it has at least a transport representative $\frF$, and by Theorem~\ref{thm:measurable} $\frF$ can be chosen to be Borel as a map from $\cP_p(\sfX)$ to $\rmT\rmL_p(\sfX;\sfY)$. This implies that we can find $\ff : \sfX \times \cP_p(\sfX) \to \sfY$ jointly Borel such that $\frF(\mu) = (\mu,\ff(\cdot,\mu))$, see Remark~\ref{rmk:mes_cont}.
\end{proof}

We also underline the following important consequence of our theory: a continuous transport representative, if it exists, is unique. Thus, there is no leeway in trying to reconstruct a continuous transport representative, as there is at most one.

\begin{theorem}
\label{thm:uniquenss_t_rep}
Take $F : \cP_p(\sfX) \to \cP_p(\sfY)$ a continuous map. Then $F$ admits at most one continuous transport representative, and it can be characterized by the values of $F$ on empirical measures.
\end{theorem}

\begin{proof}
Suppose $\frF$, $\frF'$ are two continuous transport representatives. Then, if $\mu$ is a generic measure, from Corollary~\ref{crl:atmost_one_f_sg} we obtain $\frF(\mu) = \frF'(\mu)$. As generic measures are dense (Corollary~\ref{crl:dense_supergeneric}) and $\frF$, $\frF'$ are continuous, we have $\frF = \frF'$ everywhere. 
\end{proof}

To conclude this section we provide two examples illustrating the sharpness of the result: even if $F$ is continuous the representative $\ff$ is not necessarily continuous in $\mu$ nor in $x$. 

\begin{example}
\label{ex:t_rep_non_c}
The transformation $F$ being continuous and the transport representative existing is not sufficient to guarantee that the transport representative can be chosen continuous as a function of $\mu$. 

We take $\sfX = \sfY = [0,1]$ and denote by $\lambda$ the Lebesgue measure on $\sfX$. Let $g : \R \to [0,1]$ a $1$-periodic and Lipschitz function such that $g_\sharp \lambda = \lambda$, for instance $g(x) = 2 \min(x,1-x)$. We also fix $\alpha \in (0,1)$. With $W_p$ the $p$-Wasserstein distance, we define 
\begin{equation*}
\ff(x,\mu) = \begin{cases}
\displaystyle{g \left( \frac{x}{W_p(\mu,\lambda)^\alpha} \right)} & \text{if } \mu \neq \lambda, \\
x & \text{if } \mu = \lambda.
\end{cases}
\end{equation*}
We also define $F(\mu) = \ff(\cdot,\mu)_\sharp \mu$, in particular $F(\lambda) = \lambda$.

\uline{$F$ has a transport representative}. Given how it is defined, it is clear that $F$ has a transport representative, namely $\ff$. 

\uline{$F$ is continuous}. Continuity in $\mu \neq \lambda$ is easy. Indeed, if $(\mu_n)_{n \in \N}$ converges to $\mu \neq \lambda$, then $\ff(\cdot,\mu_n)$ converges uniformly to $\ff(\cdot,\mu)$ on $\sfX$, thus the result follows by Lemma~\ref{lm:joint_continuity_fsharpmu}.  

Let us also prove continuity at $\lambda$. If $\mu \in \cP_p(\sfX)$, by the triangle inequality
\begin{equation*}
W_p(F(\mu),F(\lambda)) = W_p(\ff(\cdot,\mu)_\sharp \mu,\lambda) \le W_p(\ff(\cdot,\mu)_\sharp \mu, \ff(\cdot,\mu)_\sharp \lambda) + W_p(\ff(\cdot,\mu)_\sharp \lambda,\lambda).
\end{equation*}
For the first term, with $L$ the Lipschitz constant of $g$, we use that $\ff(\cdot,\mu)$ is $L W_p(\mu,\lambda)^{-\alpha}$ Lipschitz. Thus with~\eqref{eq:lift_lipschitz} we obtain
\begin{equation*}
 W_p(\ff(\cdot,\mu)_\sharp \mu, \ff(\cdot,\mu)_\sharp \lambda) \le  L W_p(\mu,\lambda)^{-\alpha} \cdot W_p(\mu,\lambda) = L W_p(\mu,\lambda)^{1-\alpha}.  
\end{equation*}
For the second term, call $a = W_p(\mu,\lambda)^\alpha$, and $n$ the largest integer with $n a \le 1$. As $\ff(\cdot,\mu)$ is $a$ periodic, $\ff(\cdot,\mu)_\sharp (\lambda \mres [0,n a]) = n a \lambda$. With $\| \cdot \|_{\mathrm{TV}}$ the total variation norm, using that it dominates the Wasserstein distance as we are on a space of diameter $1$, 
\begin{equation*}
W_p(\ff(\cdot,\mu)_\sharp \lambda,\lambda) \le \| \ff(\cdot,\mu)_\sharp \lambda - \lambda \|_\mathrm{TV} \leq \| \ff(\cdot,\mu)_\sharp (\lambda \mres [0,na]) - \lambda \|_\mathrm{TV} + \| \ff(\cdot,\mu)_\sharp (\lambda \mres [na,1])  \|_\mathrm{TV} \leq 2 a 
\end{equation*}
where we used $1 - na < a$. Given the expression of $a$,
putting these two pieces together, we deduce, with $L$ the Lipschitz constant of $g$,
\begin{equation*}
W_p(F(\mu),F(\lambda)) \leq L W_p(\mu,\lambda)^{1-\alpha} + 2 W_p(\mu,\lambda)^\alpha. 
\end{equation*}
This is enough to conclude to continuity of $F$ at $\lambda$. 

\uline{$F$ cannot have a continuous transport representative}.
If $F$ admitted a continuous transport representative $\frF$ then by Corollary~\ref{crl:atmost_one_f_sg} and continuity of $\mu \mapsto \ff(\cdot,\mu)$, we have $\frF(\mu) = (\mu, \ff(\cdot,\mu))$ at least for $\mu \neq \lambda$. Being $h$ a bounded continuous function with $\int h \d \lambda = 0$, we consider $\mu_\varepsilon = (1+\varepsilon h) \lambda$ which belongs to $\cP(\sfX)$ for $\varepsilon$ small enough. Moreover $\mu_\varepsilon \to \lambda$ as $\varepsilon \to 0$, and it can be checked that $\iota(\mu_\varepsilon,\ff(\cdot,\mu_\varepsilon)) = (\ii \times \ff(\cdot,\mu_\varepsilon))_\sharp \mu_\varepsilon$ converges narrowly to $\lambda \otimes \lambda$ which is not an element of $\cPg_p(\sfX \times \sfY)$.
Thus the family $\{ (\mu_\varepsilon, \ff(\cdot,\mu_\varepsilon)) \ : \ \varepsilon > 0 \}$ cannot be relatively compact. But if $\frF$ is continuous then $\frF(\cP_p(\sfX))$ is compact as the image of a compact set by a continuous map.
\end{example}

\begin{example}
\label{ex:continuity_x}
The transformation $F$ being continuous and the transport representative existing is not sufficient to guarantee that the transport representative can be chosen continuous as a function of $x$ for a fixed $\mu$.

We take $\sfX = \sfY = [-1,1]$ and denote by $\lambda$ the Lebesgue measure on $\sfX$, rescaled to be a probability measure. With $W_p$ the $p$-Wasserstein distance, and $\mathrm{sgn}$ the sign function, we define 
\begin{equation*}
\ff(x,\mu) = \begin{cases}
\displaystyle{\tanh \left( \frac{x}{W_p(\mu,\lambda)} \right)} & \text{if } \mu \neq \lambda, \\
\mathrm{sgn}(x) & \text{if } \mu = \lambda.
\end{cases}
\end{equation*}
We also define $F(\mu) = \ff(\cdot,\mu)_\sharp \mu$, in particular $F(\lambda) = \frac{1}{2} (\delta_{-1} + \delta_{1})$.

\uline{$F$ has a continuous transport representative}. Clearly $F$ has a transport representative that we write $\frF$, we claim that $\mu \mapsto  (\ii \times \ff(\cdot,\mu))_\sharp \mu$ is continuous, so that $\frF$ is also continuous, thus $F$ too.

Continuity in $\mu \neq \lambda$ is easy. Indeed, if $\mu_n$ converges to $\mu \neq \lambda$ as $n \to + \infty$, then $\ff(\cdot,\mu_n)$ converges uniformly to $\ff(\cdot,\mu)$ on $\sfX$, thus we can apply Lemma~\ref{lm:joint_continuity_fsharpmu}. 

Let us look at continuity at $\lambda$: we take a sequence $(\mu_n)_{n \in \N}$ converging to $\lambda$. Fix $\eps > 0$ and write $\sfX_\eps = \sfX \setminus [- \eps,\eps]$. Then $\ii \times \ff(\cdot,\mu_n)$ converges uniformly to $\ii \times \ff(\cdot,\lambda)$ on $\sfX_\eps$. As moreover $\mu_n \mres \sfX_\eps$ converges narrowly to $\lambda \mres \sfX_\eps$, we see that $(\ii \times \ff(\cdot,\mu_n))_\sharp (\mu_n \mres \sfX_\eps)$ converges to $(\ii \times \ff(\cdot,\lambda))_\sharp (\lambda \mres \sfX_\eps)$ (Lemma~\ref{lm:joint_continuity_fsharpmu}). This latter expression converges to $(\ii \times \ff(\cdot,\lambda))_\sharp \lambda$ as $\eps \to 0$. On the other hand the total variation norm of $(\ii \times \ff(\cdot,\mu_n))_\sharp (\mu_n \mres [-\eps,\eps])$ is bounded by $\eps$. As $\eps$ is arbitrary we deduce from this that $(\ii \times \ff(\cdot,\mu_n))_\sharp \mu_n =(\ii \times  \ff(\cdot,\mu_n))_\sharp (\mu_n \mres \sfX_\eps) + (\ii \times \ff(\cdot,\mu_n))_\sharp (\mu_n \mres [-\eps,\eps])$ converges narrowly to $(\ii \times \ff(\cdot,\lambda))_\sharp \lambda$.

\uline{The transport representative is not continuous as function of $x$ for $\mu=\lambda$}. Clearly $\ff(\cdot,\lambda)$ is discontinuous. Actually there is no continuous map $f$ such that $f_\sharp \lambda = F(\lambda) = \frac{1}{2} (\delta_{-1} + \delta_{1})$.
\end{example}

\subsection{Lipschitz transformation $F$}

In the previous section we assumed only that $F$ was continuous, and it is not enough to guarantee the existence of a continuous transport representative even with the non-splitting assumption. We now prove that if $F$ is Lipschitz we can prove the existence of a continuous transport representative.

This result will work in the framework of length spaces, we recall the minimal notions of analysis in metric spaces needed for our purpose: see \cite[Section 1.1]{AGS} or \cite{ambrosio2004topics} for a complete overview and proofs of the following. In a metric space $\sfZ$, a curve $t \in [0,1] \mapsto z(t) \in \sfZ$ is Lipschitz if there exists $L < + \infty$ such that $\sfd_\sfZ(z(t),z(s)) \leq L |t-s|$ for all $t,s \in [0,1]$. In this case, the \emph{metric derivative} of $z$, defined as
\begin{equation*}
|\dot{z}(t)|_{\sfZ} = \lim_{s \to t, s \neq t} \frac{\sfd_{\sfZ}(z(s),z(t))}{|s-t|}
\end{equation*}
exists for a.e. $t$. 
The length of $z$ is defined as
\begin{equation*}
\mathrm{length}(z) = \int_0^1 |\dot{z}(t)|_\sfZ \, \d t. 
\end{equation*}
Up to reparametrizating the curve, which does not change the length, we can assume that $|\dot{z}(t)|_{\sfZ} = \mathrm{length}(z)$ for a.e. $t \in [0,1]$.
The space $(\sfZ, \sfd_\sfZ)$ is called a \emph{length space} if the distance between $z_0,z_1 \in \sfZ$ is the infimum of $\mathrm{length}(z)$ for all Lipschitz curves $z : [0,1] \to \sfZ$ joining $z_0$ to $z_1$. If the infimum is attained, we call it a \emph{geodesic space}.
Any convex subset of a Euclidean space is a geodesic space, thus a length space.

\begin{theorem}
\label{thm:transport_lift_lipschitz}
Assume $\sfX$ is a length space. 
Take $F : \cP_p(\sfX) \to \cP_p(\sfY)$ a $L$-Lipschitz map (meaning $W_p(F(\mu),F(\mu')) \leq L W_p(\mu,\mu')$ for all $\mu, \mu' \in \cP_p(\sfX)$) which is non-splitting for empirical measures. 

Then $F$ admits a continuous transport representative. 
The transport representative $\frF$ is $(1+L^p)^{1/p}$-Lipschitz from $\cP_p(\sfX)$ to $\rmT\rmL_p(\sfX;\sfY)$. 
Moreover, writing $\frF(\mu) = (\mu, \ff(\cdot,\mu))$, for any $\mu, \mu' \in \cP_p(\sfX)$ and any $\eta \in \Gamma(\mu,\mu')$
\begin{equation}
\label{eq:regularity_fundamental}
\|\sfd_\sfY(\ff(x,\mu),\ff(x',\mu'))\|_{L^p(\eta)} \leq L \, \|\sfd_\sfX(x,x')\|_{L^p(\eta)}.
\end{equation}
\end{theorem}

\begin{corollary}
\label{crl:f_jointly_continuous}
Under the same assumptions as Theorem~\ref{thm:transport_lift_lipschitz} one can choose the transport representative $\ff$ which is $L$-Lipschitz on $\supp(\mu)$ for a fixed $\mu$ and such that: if $x_n \in \supp(\mu_n)$, $(x_n,\mu_n)$ converges to $(x,\mu)$ in $\sfX \times \cP_p(\sfX)$ and $x \in \supp(\mu)$ then $\ff(x_n,\mu_n)$ converges to $\ff(x,\mu)$.
\end{corollary}

The assumption $x \in \supp(\mu)$ cannot be removed in the corollary, that is, $\ff$ is not necessarily continuous over $\sfX \times \cP_p(\sfX)$: see Example~\ref{ex:lift_discontinuous_supp} below.

\begin{remark}[Relaxing the assumption of length space]
\label{rmk:length_space}
The assumption that $\sfX$ be a length space can be relaxed to the requirement that its intrinsic length distance $\sfd_\ell$, defined by $\sfd_\ell(x,x') := \inf\{ \mathrm{length}(z) : z \text{ a Lipschitz curve from } x \text{ to } x' \}$, is bi-Lipschitz equivalent to $\sfd_\sfX$; that is (the inequality $\sfd_\sfX \leq \sfd_\ell$ being always true), there exists $C \geq 1$ such that $\sfd_\ell \leq C\,\sfd_\sfX$. Indeed, in this case if $F$ is $L$-Lispchitz for the Wasserstein distance built on $(\sfX,\sfd_\sfX)$, then it is $CL$-Lipschitz for the Wasserstein distance built on the length space $(\sfX,\sfd_\ell)$. 
This includes the case of compact connected subsets of $\R^d$ with sufficiently regular boundary.
\end{remark}

The proof of this theorem requires two auxiliary results that we state and prove first. The first one is an abstract result about extension of uniformly continuous functions in non-complete spaces, here the space $\rmT\rmL_p(\sfX \times \sfY)$. The key point is to use the characterization of relatively compact sets in $\rmT\rmL_p(\sfX \times \sfY)$ to bypass the non-completeness.

\begin{proposition}
\label{prop:extension_TL}
Let $\frF : A \to \rmT\rmL_p(\sfX;\sfY)$ a map defined on $A$ a dense subset of $\cP_p(\sfX)$. Then $\frF$ can be extended in a continuous map $\bar \frF : \cP_p(\sfX) \to \rmT\rmL_p(\sfX;\sfY)$ if and only if, for every $B \subseteq A$ relatively compact set of $\cP_p(\sfX)$, the following two conditions hold:
\begin{enumerate}
\item the map $\frF$, restricted to $B$, is uniformly continuous as valued in $(\rmT\rmL_p(\sfX;\sfY), \sfd_{\rmT\rmL_p})$,
\item the set $\frF(B) = \{ \frF(\mu) \ : \ \mu \in B \}$ is equidispersed. 
\end{enumerate}
\end{proposition}

Though $\rmT\rmL_p(\sfX \times \sfY)$ is completely metrizable, note that the first point the assumption of uniform continuity is phrased in the metric $\sfd_{\rmT\rmL_p}$ which is not complete. 

\begin{proof}
Assume first that $\frF$ can be extended into a continuous map $\bar \frF$. Then, for every $B \subseteq A$ relatively compact, with $\bar{B}$ the compact closure of $B$, the map  $\bar \frF$ is uniformly continuous over $\bar B$ by Heine's theorem, thus $\frF$, restricted to $\bar B$, is uniformly continuous. This is the first point. Moreover, $\frF(B) \subseteq \bar \frF(\bar B)$ and the latter is compact as image of a compact set by a continuous map, thus $\frF(B)$ is relatively compact. The second point follows by Theorem~\ref{thm:rel_cpct_TL}. 

Conversely, assume that $\frF$ satisfies the two conditions for any relatively compact subset $B \subseteq A$. From the first point we can extend the function $\iota \circ \frF : A \to \cP_p(\sfX \times \sfY)$ into a continuous function $\frG : \cP_p(\sfX) \to \cP_p(\sfX \times \sfY)$. Indeed we simply use the usual extension of uniformly continuous functions valued in complete spaces, in particular 
\begin{equation*}
\frG(\mu) = \lim_{\mu' \to \mu, \: \mu' \in A} \iota(\frF(\mu')).
\end{equation*}
To conclude we only need to show that $\frG$ is valued in $\cPg(\sfX \times \sfY)$. Take $\mu \in \cP_p(\sfX)$ and $\mu_n$ a sequence in $A$ converging to $\mu$. As $\iota(\frF(\mu_n))$ converges, the family $(\iota(\frF(\mu_n)))_{n \in \N}$ is relatively compact in $\cP_p(\sfX \times \sfY)$. Moreover $(\frF(\mu_n))_{n \in \N}$ is equidispersed by the second assumption. By Theorem~\ref{thm:rel_cpct_TL} this family is relatively compact. We deduce that $(\frF(\mu_n))_{n \in \N}$ converges, up to extraction, to an element of $\rmT\rmL_p(\sfX;\sfY)$. As a consequence $\frG(\mu)$ belongs to $\cPg(\sfX \times \sfY)$. The conclusion follows by defining $\bar \frF = \iota^{-1} \circ \frG$.      
\end{proof}

The second result we need for the proof of our theorem is a key estimate which enables to check both conditions of Proposition~\ref{prop:extension_TL} above: it corresponds to estimate~\eqref{eq:regularity_fundamental} for generic empirical measures.

\begin{proposition}
\label{prop:estimate_discrete}
Assume $\sfX$ is a length space, 
that $F : \cP(\sfX) \to \cP(\sfY)$ is non-splitting on empirical measures and $L$-Lipschitz. Fix $\mu,\mu' \in \cP_p(\sfX)$ two generic empirical measures, $\eta \in \Gamma(\mu,\mu')$ and $f$, $f'$ two maps such that $f_\sharp \mu = F(\mu)$ and $f'_\sharp\mu' = F(\mu')$. Then 
\begin{equation*}
\| \sfd_{\sfY}(f(x),f'(x')) \|_{L^p(\eta)} \leq L \| \sfd_{\sfX}(x,x') \|_{L^p(\eta)}.
\end{equation*}
\end{proposition}

\begin{proof}
\uline{We first assume that $\eta$ is a generic empirical measure}, we write it
$\eta = \sum_{i=1}^m a_i \delta_{(x_i,x'_i)}$,
where we may have repetitions between $x_i$ and $x'_i$ but $(a_i)_{1 \leq i \leq m}$ is generic. We fix $\alpha > 0$.
For each $i$, as $\sfX$ is a length space, we can find $(x_i(t))_{t \in [0,1]}$ a curve joining $x_i$ to $x_i'$ with metric derivative bounded as $|\dot{x}_i(t)|_{\sfX} \leq \sfd_\sfX(x_i,x'_i) + \alpha$. We write $\mu(t) = \sum_{i=1}^m a_i \delta_{x_i(t)}$. With the notations of Section~\ref{sec:generic_measures} we have $\mu(t) = \frG_{\aa}(\xx(t))$ where $\xx(t) = (x_1(1), \ldots, x_m(t))$. Using~\eqref{eq:estimate_trivial} we can easily check that $t \mapsto \mu(t)$ is Lipschitz.
Moreover with Proposition~\ref{proposition_local_isometry}, we obtain that the curve $t \mapsto \mu(t)$, valued in $\cP_p(\sfX)$, has metric derivative coinciding with the one of $t \mapsto \xx(t)$ in the space $(\sfX^m, \sfd_{\aa, \sfX^m})$: for a.e. $t$
\begin{equation*}
\big|\dot\mu(t) \big|_{\cP_p(\sfX)} = \| \dot{\xx}(t) \|_{\sfX^m}  \leq \left( \sum_{i=1}^m a_i (\sfd_\sfX(x_i,x'_i) + \alpha )^p \right)^{1/p} \leq  \| \sfd_\sfX(x,x') \|_{L^p(\eta)} + \alpha.
\end{equation*}
We look at $\nu(t) = F(\mu(t))$. As $F$ is $L$-Lipschitz the curve $t \mapsto \nu(t)$ is also Lipschitz, with metric derivative satisfying a.e. 
\begin{equation*}
\big|\dot\nu(t) \big|_{\cP_p(\sfY)} \leq L \big|\dot\mu(t) \big|_{\cP_p(\sfX)} \leq L ( \| \sfd_\sfX(x,x') \|_{L^p(\eta)} + \alpha)
\end{equation*} 
Moreover $F$ is non-splitting on empirical measures, thus $\nu(t) = \sum_{i=1}^m a_i \delta_{y_i(t)}$. Moreover, as $\aa$ is generic, by Proposition~\ref{prop:super_generic_injective} we deduce that $y_i(t)$ is a continuous function of $t$. As $\nu(0) = F(\mu) = f_\sharp \mu = \sum_{i=1}^m a_i \delta_{f(x_i)}$ and $\nu(1) = F(\mu') = f'_\sharp \mu' = \sum_{i=1}^m a_i \delta_{f'(x'_i)}$ we have necessarily $y_i(0) = f(x_i)$ and $y_i(1) = f'(x_i')$. 

We write $\yy(t) = (y_1(t), y_2(t), \ldots, y_m(t))$.  
For any $t \in [0,1]$, by continuity of $t \mapsto \yy(t)$ and Proposition~\ref{proposition_local_isometry}, we can find $\varepsilon > 0$ (depending on $t$) such that if $|s-t| < \varepsilon$ then 
\begin{equation}
\label{eq:aux_technical}
\sfd_{\aa, \sfY^m}(\yy(t), \yy(s)) = W_p(\nu(t), \nu(s)) \leq L |t-s| ( \| \sfd_\sfX(x,x') \|_{L^p(\eta)} + \alpha).
\end{equation}
By compactness of $[0,1]$, we can find $t_1, \ldots t_\ell$ and $\varepsilon_1, \ldots, \varepsilon_\ell$ such that $[t_k - \varepsilon_k, t_k + \varepsilon_k]$ for $k = 1, \ldots, \ell$ is a covering of $[0,1]$ and~\eqref{eq:aux_technical} holds if $t=t_k$ and $|s- t_k| < \varepsilon_k$. Without loss of generality we can always assume $0= t_1 \leq t_2 \leq \ldots \leq t_\ell = 1$, so that $t_{k+1} - t_k < \varepsilon_k$. Then we use the triangle inequality to get the result:
\begin{align*}
\| \sfd_{\sfY}(f(x),f'(x')) \|_{L^p(\eta)}  = 
\sfd_{\aa, \sfY^m}(\yy(0), \yy(1))  & \leq 
\sum_{k=1}^{\ell - 1} \sfd_{\aa, \sfY^m}(\yy(t_k), \yy(t_{k+1}))  \\ 
& \leq L \sum_{k=1}^{\ell - 1} (t_{k+1} - t_k) ( \| \sfd_\sfX(x,x') \|_{L^p(\eta)} + \alpha) \\
& = L\| \sfd_{\sfX}(x,x') \|_{L^p(\eta)} + L \alpha.
\end{align*}
As $\alpha > 0$ is arbitrary we obtain our conclusion.

\medskip

\uline{In the case $\eta$ is not generic}, we proceed by approximation. We write again $\eta = \sum_{i=1}^m a_i \delta_{(x_i,x'_i)}$. By Proposition~\ref{prop:super_generic_dense} we can find a sequence $\aa^{(n)}$ of generic rational weights converging to $\aa$. We write $\eta_n = \sum_{i=1}^m a_{i}^{(n)} \delta_{(x_i,x'_i)}$ and $\mu_n, \mu'_n$ for the marginals of $\eta_n$. We can easily check that $\mu_n$ and $\mu'_n$ are generic empirical measures. As $F$ is non-splitting, there exist $f_n, f'_n$ such that $F(\mu_n) = {f_n}_\sharp \mu_n$ and $F(\mu'_n) = {f'_n}_\sharp \mu'_n$. From the previous case we already know that $\| \sfd_{\sfY}(f_n(x),f_n'(x')) \|_{L^p(\eta_n)} \leq L \| \sfd_{\sfX}(x,x') \|_{L^p(\eta_n)}$. As $\eta_n$ has finite support, the same as the one of $\eta$, it remains to justify that $f_n(x_i) \to f(x_i)$ and $f'_n(x'_i) \to f'(x_i)$ as $n \to + \infty$. 

By continuity of $F$, the sequences $(F(\mu_n))_{n \in \N}$ and $(F(\mu'_n))_{n \in \N}$ are tight, so that, for every $i$, the sequences $(f_n(x_i))_{n \in \N}$ and $(f'_n(x'_i))_{n \in \N}$ belong to a compact set. Let us call $y_i$ any limit point of the sequence $(f_n(x_i))_{n \in \N}$. By continuity of $F$ the map sending $x_i$ to $y_i$ for all $i$ is a transport map between $\mu$ and $F(\mu)$. As $\mu$ is generic, by Corollary~\ref{crl:atmost_one_f_sg}, necessarily $y_i = f(x_i)$. Thus there is at most one limit point, meaning the sequence $(f_n(x_i))_{n \in \N}$ converges to $f(x_i)$. A similar reasoning shows that $f'_n(x'_i) \to f'(x_i)$ as $n \to + \infty$. The conclusion follows.    
\end{proof}

\begin{proof}[Proof of Theorem~\ref{thm:transport_lift_lipschitz}]
\uline{Existence of $\frF$.}
We rely on Proposition~\ref{prop:extension_TL} with $A$ the set of generic empirical measures. 

Indeed, if $\mu \in \cP_p(\sfX)$ is a generic empirical measure, then as $F$ is non splitting we find $\ff(\cdot,\mu)$ such that $\ff(\cdot,\mu)_\sharp \mu = F(\mu)$, and by Corollary~\ref{crl:atmost_one_f_sg} there is only one such $\ff(\cdot,\mu)$. We define $\frF(\mu) = (\mu,\ff(\cdot,\mu))$ for any such $\mu$ empirical generic.

We first claim that $\frF$ is $(1+L^p)^{1/p}$-Lipschitz on $A$. Indeed, let $\mu, \mu' \in A$ and $\eta \in \Gamma(\mu,\mu')$ an optimal transport coupling between $\mu$ and $\mu'$. Then $(\pi^1 \times \ff(\pi^1,\mu) \times \pi^2 \times \ff(\pi^2,\mu'))_\sharp \eta$ is a coupling between $\iota(\frF(\mu))$ and $\iota(\frF(\mu'))$, see Lemma~\ref{lm:couplings_sigma_eta}. Thus with Proposition~\ref{prop:estimate_discrete},
\begin{align*}
\sfd_{\rmT\rmL_p}(\frF(\mu),\frF(\mu'))^p & \leq \| \sfd_\sfX(x,x') \|^p_{L^p(\eta)} + \| \sfd_\sfY(\ff(x,\mu),\ff(x',\mu')) \|^p_{L^p(\eta)} \\
& \leq (1+L^p) \| \sfd_\sfX(x,x') \|^p_{L^p(\eta)}  = (1+L^p) W_p(\mu,\mu')^p. 
\end{align*}

Moreover, the family $\frF(A)$ is equidispersed: if $\eta \in \Gamma(\mu,\mu)$ then with Proposition~\ref{prop:estimate_discrete} we have 
\begin{equation*}
\| \sfd_{\sfY}(\ff(x,\mu),\ff(x',\mu)) \|_{L^p(\eta)} \leq L \| \sfd_{\sfX}(x,x') \|_{L^p(\eta)}.   
\end{equation*}
Thus it shows that a modulus of equidispersivity for the family $\frF(A)$ is $\omega(s) = Ls$. 

We conclude with Proposition~\ref{prop:extension_TL} that $\frF$ can be extended into a Lipschitz map $\bar \frF : \cP_p(\sfX) \to \rmT\rmL_p(\sfX;\sfY)$. For any $\mu \in A$ we have $\iota(\frF(\mu)) \in \Gamma(\mu,F(\mu))$. By density, as $F$ is continuous and $\Gamma$ is hemicontinuous (Lemma~\ref{lm:Gamma_hemic}), $\iota(\bar \frF(\mu)) \in \Gamma(\mu,F(\mu))$ for all $\mu \in \cP_p(\sfX)$. It shows that $\bar \frF$ is a transport representative of $F$, see \eqref{eq:transportlift_char}.

\uline{Regularity of $\frF$.}
We have already seen that $\frF$ is $(1+L^p)^{1/p}$-Lipschitz on $A$, thus clearly its extension is $(1+L^p)^{1/p}$-Lipschitz on the closure of $A$, that is, $\cP_p(\sfX)$.
We also need to prove~\eqref{eq:regularity_fundamental}, for this we use of course Proposition~\ref{prop:estimate_discrete}. Let us fix $\mu, \mu' \in \cP_p(\sfX)$, and a coupling $\eta \in \Gamma(\mu,\mu')$.
By Corollary~\ref{crl:dense_supergeneric} we can approximate $\mu, \mu'$ by sequences $\mu_n, \mu'_n$ of generic empirical measures converging in $\cP_p(\sfX)$. By the lower hemicontinuity of $\Gamma$ (Lemma~\ref{lm:Gamma_hemic}), there exists a sequence $\eta_n \in \Gamma(\mu_n,\mu'_n)$ converging to $\eta$ in $\cP_p(\sfX \times \sfX)$. By Proposition~\ref{prop:estimate_discrete}, for each $n$ we have
\begin{equation*}
\|\sfd_\sfY(\ff(x,\mu_n),\ff(x',\mu'_n))\|_{L^p(\eta_n)} \leq L \, \|\sfd_\sfX(x,x')\|_{L^p(\eta_n)}.
\end{equation*}
The right-hand side converges to $L\,\|\sfd_\sfX(x,x')\|_{L^p(\eta)}$. For the left-hand side, by continuity of the transport representative $\mu \mapsto (\mu,\ff(\cdot,\mu))$ (Theorem~\ref{thm:transport_lift_lipschitz}) and the characterization of $\rmT\rmL_p$ convergence in Proposition~\ref{prop:char_conv_TLp}, the integrals pass to the limit. We obtain the claimed estimate.
\end{proof}

\begin{proof}[Proof of Corollary~\ref{crl:f_jointly_continuous}]
\uline{The representative $\ff$ is $L$-Lipschitz on $\supp(\mu)$}. We fix $\mu \in \cP_p(\sfX)$, choose a representative $f$ of $\ff(\cdot,\mu)$, and we write $\gamma = (\ii \times f)_\sharp \mu$. We want to prove that, up to changing the value of $f$ on a negligible set, we can make it $L$-Lipschitz. We can find a negligible set $N \subset \sfX$ with $\mu(N) = 0$ such that, if $x \notin N$ then $(x,f(x)) \in \supp(\gamma)$. Take $x_0 \neq x_1$ two points not in $N$, and fix $r \in (0,\sfd_\sfX(x_0,x_1)/2)$. For $i=0,1$ we look at $S_{i,r}$ the set of $x$ such that $\sfd_\sfX(x_i,x) \leq r$ and $\sfd_\sfY(f(x_i),f(x)) \leq r$. As $(x_i,f(x_i))$ is in the support of $\gamma$ we have $\mu(S_{i,r}) > 0$ for all $r > 0$. We consider the coupling $\eta$ which ``swaps'' $S_{0,r}$ and $S_{1,r}$: assuming $\mu(S_{0,r}) \leq \mu(S_{1,r})$, with $\mu_{i,r} = \mu \mres S_{i,r}$ the measure $\mu$ restricted to $S_{i,r}$ and $\ii : \sfX \to \sfX$ the identity map,   
\begin{equation}
\label{eq:eta_swap}
\eta = (\ii \times \ii)_\sharp \left(\mu - \mu_{0,r} - \frac{\mu(S_{0,r})}{\mu(S_{1,r})}  \mu_{1,r} \right) + \frac{1}{\mu(S_{1,r})} \mu_{0,r} \otimes \mu_{1,r} + \frac{1}{\mu(S_{1,r})} \mu_{1,r} \otimes \mu_{0,r}.
\end{equation}
Given how it is constructed, we see that $\eta \in \Gamma(\mu,\mu)$. 
We use the estimate~\eqref{eq:regularity_fundamental} with $\mu = \mu'$ and this choice of $\eta$: we obtain
\begin{align*}
\frac{2}{\mu(S_{1,r})} \iint_{S_{0,r} \times S_{1,r}} \sfd_\sfY(f(x),f(x'))^p \, \d \mu(x) \d \mu(x')  & =  \| \sfd_{\sfY}(f(x),f(x')) \|^p_{L^p(\eta)} \\
& \leq L^p \| \sfd_{\sfX}(x,x') \|^p_{L^p(\eta)} \\
& = \frac{2 L^p}{\mu(S_{1,r})} \iint_{S_{0,r} \times S_{1,r}} \sfd_\sfX(x,x')^p \, \d \mu(x) \d \mu(x').
\end{align*}
Given how the sets are constructed,  if $(x,x') \in S_{0,r} \times S_{1,r}$ then $| \sfd_\sfY(f(x),f(x')) - \sfd_\sfY(f(x_0),f(x_1)) | \leq 2 r$ and $|\sfd_\sfX(x,x') - \sfd_\sfX(x_0,x_1) | \leq 2 r$. Therefore we obtain
\begin{align*}
\sfd_\sfY(f(x_0),f(x_1))^p & = \lim_{r \to 0} \, \frac{1}{\mu(S_{0,r}) \mu(S_{1,r})} \iint_{S_{0,r} \times S_{1,r}} \sfd_\sfY(f(x),f(x'))^p \, \d \mu(x) \d \mu(x') \\
& \leq L^p \, \limsup_{r \to 0} \, \frac{1}{\mu(S_{0,r}) \mu(S_{1,r})} \iint_{S_{0,r} \times S_{1,r}} \sfd_\sfX(x,x')^p \, \d \mu(x) \d \mu(x') = L^p \sfd_\sfX(x_0,x_1)^p. 
\end{align*}
It reads $\sfd_\sfY(f(x_0),f(x_1)) \leq L \sfd_\sfX(x_0,x_1)$: as this is valid outside of $N$, that is, for $\mu$-a.e. $x_0,x_1$, up to modifying $f$ on a null set, we see that $f$ is $L$-Lipschitz on $\supp(\mu)$.

\uline{Joint continuity in $(x,\mu)$.}
Choose for $\ff(\cdot,\mu)$ a representative which is $L$-Lipschitz on $\supp(\mu)$. Then let us consider $(x_n,\mu_n)$ a sequence converging in $\sfX \times \cP_p(\sfX)$ to $(x,\mu)$ with $x_n \in \supp(\mu_n)$ and $x \in \supp(\mu)$. 
We will use Lusin's theorem for $\rmT\rmL_p(\sfX;\sfY)$ convergence (Theorem~\ref{thm:Gconv_TLpconv}) which helps us to go back to a convergence of continuous functions with compact graphs. This theorem requires bounded distances, but we can always consider bounded distances on $\sfX$ and $\sfY$ which generate the same topology. 
Fix $r>0$, and let $0 < \eps < \mu(B(x,r/2))/2$. We use Theorem~\ref{thm:Gconv_TLpconv} with $f_n = \ff(\cdot,\mu_n)$ and $f = \ff(\cdot,\mu)$: let $\mu_\eps$, $\mu_{\eps,n}$, $K_\eps$ and $K_{\eps,n}$ as in the statement of this theorem. As $\lim_{n \to + \infty} \mu_n(B(x_n,r)) \geq \mu(B(x,r/2)) > 2 \eps$, for $n$ large enough $K_{\eps,n}$ intersects $B(x_n,r)$. Let us call $x'_n$ a point in the intersection, we can always choose it in $\supp(\mu_n)$. By the $\rmG$-convergence of $\ff(\cdot,\mu_n) \restr{K_{\eps,n}}$ to $\ff(\cdot,\mu) \restr{K_{\eps}}$, up to extraction $x'_n$ converges to a limit $x' \in K_\eps$ and $\ff(x'_n,\mu_n)$ converges to $\ff(x',\mu)$. As necessarily $x' \in B(x,r)$ we see by Lipschitzianity that 
\begin{align*}
\sfd_\sfY(\ff(x_n,\mu_n),\ff(x,\mu)) & \leq \sfd_\sfY(\ff(x_n,\mu_n),\ff(x'_n,\mu_n))  + \sfd_\sfY(\ff(x'_n,\mu_n),\ff(x',\mu)) + \sfd_\sfY(\ff(x',\mu),\ff(x,\mu))   \\
& \leq Lr + \sfd_\sfY(\ff(x'_n,\mu_n),\ff(x',\mu)) + Lr,
\end{align*}
and the second term converges to $0$ as $n \to + \infty$. 
Thus we have proven: for every $r > 0$, there exists an extraction $k \mapsto n_k$ such that 
\begin{equation*}
\limsup_{k \to + \infty} \sfd_\sfY(\ff(x_{n_k},\mu_{n_k}),\ff(x,\mu)) \leq 2Lr. 
\end{equation*}
This is enough to conclude that $\ff(x_n,\mu_n)$ converges to $\ff(x,\mu)$: if it were not the case then along a subsequence $\sfd_\sfY(\ff(x_n,\mu_n),\ff(x,\mu))$ would converge to a strictly positive limit.    
\end{proof}

We now prove our second main result as it was stated in the introduction, it is an easy consequence of what is above.

\begin{proof}[\textbf{Proof of Theorem~\ref{thm:main_Lipschitz}}]
This is an easy consequence of Theorem~\ref{thm:transport_lift_lipschitz} and Corollary~\ref{crl:f_jointly_continuous}. The only additional thing we have to justify is that we can choose $\ff$ jointly Borel on $\sfX \times \cP_p(\sfX)$. We define $\ff(\cdot,\mu)$ on $\supp(\mu)$ as the representative which is $L$-Lipschitz; while if $x \notin \supp(\mu)$ we define $\ff(x,\mu) = y_0$ for some fix $y_0 \in \sfY$.

As already noted in \cite{fornasier2023density} at least for $\sfX = \R^d$, the set $\{ (x,\mu) \ : \ x \in \supp(\mu)  \}$ is a $G_\delta$, in particular Borel, and dense subset of $\sfX \times \cP_p(\sfX)$. Let us give a self-contained argument valid for any Polish space $\sfX$. For every $r > 0$, the map $(x,\mu) \mapsto \mu(B(x,r))$ is lower semicontinuous on $\sfX \times \cP_p(\sfX)$. Indeed, if $x_n \to x$ and $\mu_n \to \mu$, then for every $s \in (0,r)$ the ball $B(x_n,r)$ eventually contains $B(x,s)$, so that
\begin{equation*}
\liminf_{n \to + \infty} \mu_n(B(x_n,r)) \geq \liminf_{n \to + \infty} \mu_n(B(x,s)) \geq \mu(B(x,s)),
\end{equation*}
where the last inequality follows the lower semicontinuity of $\nu \mapsto \nu(U)$ on open sets $U$ for the topology of narrow convergence. Letting $s \to r$ yields $\liminf_{n \to + \infty} \mu_n(B(x_n,r)) \geq \mu(B(x,r))$. Using the characterization $x \in \supp(\mu)$ if and only if $\mu(B(x,r)) > 0$ for all $r > 0$, we get
\begin{equation*}
\{ (x,\mu) \ : \ x \in \supp(\mu) \} = \bigcap_{n \geq 1} \{ (x,\mu) \ : \ \mu(B(x,1/n)) > 0 \},
\end{equation*}
a countable intersection of open sets, hence a $G_\delta$. Density follows from the fact that, for every $(x,\mu) \in \sfX \times \cP_p(\sfX)$ and every $\alpha \in (0,1]$, the measure $(1-\alpha)\mu + \alpha\delta_x$ contains $x$ in its support and converges to $\mu$ in $\cP_p(\sfX)$ as $\alpha \to 0$.

The conclusion follows: the function $\ff$ is Borel as it coincides with a continuous function on a Borel set (by Corollary~\ref{crl:f_jointly_continuous}), and is constant on the complement of this Borel set. 
\end{proof}

\begin{example}
\label{ex:lift_discontinuous_supp}
Corollary~\ref{crl:f_jointly_continuous} shows that $\ff(x_n,\mu_n) \to \ff(x,\mu)$ as soon as $(x_n,\mu_n) \to (x,\mu)$ and $x_n \in \supp(\mu_n)$, $x \in \supp(\mu)$. We show that the latter condition cannot be removed: the unique continuous transport representative cannot necessarily be chosen continuous uniformly over $\sfX \times \cP_p(\sfX)$.

We take $\sfX = \sfY = \R$ and $p=1$. We choose $\ff(\cdot,\mu)$ linear: specifically for $\mu \neq \delta_0$ we define
\begin{equation*}
s(\mu) = \frac{\int x \d \mu(x)}{\int |x| \d \mu(x)} = \frac{\int x \d \mu(x)}{W_1(\mu,\delta_0)},
\end{equation*}
and then $\ff(x,\mu) = s(\mu) x$ as well as $F(\mu) = \ff(\cdot,\mu)_\sharp \mu$. If $\mu = \delta_0$ we simply define $\ff(\cdot,\delta_0) = 0$, thus $F(\delta_0) = \delta_0$. Note that the map $s : \cP_1(\sfX) \to \R$ is discontinuous at $\delta_0$: indeed $s(\delta_{x}) = \mathrm{sgn}(x)$. Nevertheless we will show that $F$ is Lipschitz. 

\uline{The map $F$ is $3$-Lipschitz}. Take $\mu, \mu' \in \cP_1(\sfX)$. Using the easy estimates~\eqref{eq:Wp_Lp} and~\eqref{eq:lift_lipschitz} (as $\ff(\cdot,\mu)$ is always $1$-Lipschitz), we have 
\begin{align*}
W_1(F(\mu),F(\mu'))  = W_1(\ff(\cdot,\mu)_\sharp \mu, \ff(\cdot,\mu')_\sharp \mu') & \leq W_1(\ff(\cdot,\mu)_\sharp \mu, \ff(\cdot,\mu)_\sharp \mu') + W_1(\ff(\cdot,\mu)_\sharp \mu', \ff(\cdot,\mu')_\sharp \mu') \\
& \leq W_1(\mu,\mu') + \| \ff(\cdot,\mu) - \ff(\cdot,\mu') \|_{L^1(\mu')} \\
& = W_1(\mu,\mu') + |s(\mu) - s(\mu')| \int |x| \d \mu'(x) \\
& = W_1(\mu,\mu') + W_1(\mu',\delta_0) |s(\mu) - s(\mu')|. 
\end{align*}
Next we decompose:
\begin{align*}
W_1(\mu',\delta_0) (s(\mu) - s(\mu')) & = \frac{W_1(\mu',\delta_0)}{W_1(\mu,\delta_0)} \int x \d \mu(x) - \int x \d \mu'(x) \\
& = \frac{W_1(\mu',\delta_0) - W_1(\mu,\delta_0)}{W_1(\mu,\delta_0)} \int x \d \mu(x) + \int x \d (\mu - \mu')(x). 
\end{align*}
As $\nu \mapsto W_1(\nu,\delta_0)$ and $\nu \mapsto \int x \d \nu(x)$ are both $1$-Lipschitz with respect to $W_1$, we see that 
\begin{equation*}
W_1(\mu',\delta_0) |s(\mu) - s(\mu')|  \leq W_1(\mu,\mu') \frac{\left|\int x \d \mu(x) \right|}{W_1(\mu,\delta_0)} + W_1(\mu,\mu') \leq 2 W_1(\mu,\mu').
\end{equation*}
Plugging in our previous estimate, we have $W_1(F(\mu),F(\mu')) \leq 3 W_1(\mu,\mu')$, which is our claim.

\uline{The transport representative cannot be chosen continuous over $\sfX \times \cP_p(\sfX)$}.
Let us define
\begin{equation*}
\mu^+_n = \frac{1}{n} \delta_1 + \left( 1- \frac{1}{n} \right) \delta_{1/\sqrt{n}}, \qquad \mu^-_n = \frac{1}{n} \delta_1 + \left( 1- \frac{1}{n} \right) \delta_{-1/\sqrt{n}}.    
\end{equation*}
These two empirical measures converge to $\delta_0$ and are generic for $n \geq 3$, with $s(\mu^+_n) = 1$ while $s(\mu^-_n) \to - 1$. Recalling that the value of any continuous transport representative is unique on generic measures (Corollary~\ref{crl:atmost_one_f_sg}), we deduce that if $\ff$ is any continuous transport representative of $F$ then $\ff(1,\mu^+_n) = 1$ while $\ff(1,\mu^-_n) \to -1$. That proves $\ff$ cannot be extended into a continuous function at $(1,\delta_0)$. 

Nevertheless it does not contradict Corollary~\ref{crl:f_jointly_continuous} as $1 \notin \supp(\delta_0) = \lim_n \mu^\pm$.
\end{example}

Corollary~\ref{crl:f_jointly_continuous} establishes the joint continuity of $\ff$ at points $(x,\mu)$ with $x \in \supp(\mu)$, but gives no quantitative information on the modulus of continuity with respect to $\mu$. We now complement this result with a quantitative stability estimate that shows: when $\mu, \mu'$ are bounded below by a fixed reference measure $\lambda$, the map $\ff(\cdot,\mu)$ varies continuously in $\mu$ in the supremum norm on $\supp(\lambda)$, with a modulus of continuity governed by how the mass $\lambda(B(x,r))$ behaves as $r \to 0$. The proof is a variant of the argument used in
the first part of the proof of Corollary~\ref{crl:f_jointly_continuous} which leverages the estimate~\eqref{eq:regularity_fundamental}.

\begin{lemma}
\label{lm:uniform_continuity_supp}
Assume $\sfX$ is a length space and $F : \cP_p(\sfX) \to \cP_p(\sfY)$ is non-splitting on empirical measures and $L$-Lipschitz. Let $\ff$ be the transport representative of $F$ given by Theorem~\ref{thm:transport_lift_lipschitz}. Let $\lambda \in \cM(\sfX)$ be a finite positive Borel measure.

Then for every $\mu, \mu' \in \cP_p(\sfX)$ with $\mu \geq \lambda$ and $\mu' \geq \lambda$ (in particular $\lambda$ has finite $p$-moment) and every $r > 0$,
\begin{equation*}
\sfd_\sfY(\ff(x,\mu),\ff(x,\mu'))^p \leq C(p) \, L^p \left( r^p + \frac{W_p(\mu,\mu')^p}{\lambda(B(x,r))} \right) \qquad \text{for every } x \in \supp(\lambda),
\end{equation*}
where $C(p)$ is a constant depending only on $p$.
\end{lemma}

\begin{proof}
Fix $x_0 \in \supp(\lambda)$ and $r > 0$, and set $S := B(x_0,r)$. Since $x_0 \in \supp(\lambda)$, $\lambda(S) > 0$.

\uline{Construction of a coupling.} Let $\eta_0 \in \Gamma(\mu,\mu')$ be an optimal transport plan, so that $\|\sfd_\sfX(x,x')\|_{L^p(\eta_0)} = W_p(\mu,\mu')$. The assumption $\mu \geq \lambda$ implies that the density $\frac{\d \lambda}{\d \mu}$ is well defined and bounded by $1$. We set
\begin{equation*}
a(x) := \frac{\d \lambda}{\d \mu}(x) \, \mathbf{1}_S(x), \qquad a'(x') := \frac{\d \lambda}{\d \mu'}(x') \, \mathbf{1}_S(x'),
\end{equation*}
so that $0 \leq a, a' \leq 1$ and $a\,\mu = \lambda\mres S = a'\,\mu'$. We then split $\eta_0$ as
\begin{equation*}
\eta_1 := \tfrac{1}{2}\, a(x)\,\eta_0, \qquad \eta_2 := \tfrac{1}{2}\, a'(x')\,\eta_0, \qquad \eta_3 := \eta_0 - \eta_1 - \eta_2,
\end{equation*}
which is a non-negative decomposition since $\frac{1}{2}(a(x) + a'(x')) \leq 1$. By construction, the marginals of $\eta_1$ are
\begin{equation*}
\pi^1_\sharp \eta_1 = \tfrac{1}{2}\,\lambda \mres S, \qquad \pi^2_\sharp \eta_1 =: \nu_1,
\end{equation*}
and the marginals of $\eta_2$ are
\begin{equation*}
\pi^1_\sharp \eta_2 =: \nu_2, \qquad \pi^2_\sharp \eta_2 = \tfrac{1}{2}\,\lambda \mres S.
\end{equation*}
Both $\nu_1$ and $\nu_2$ have the same total mass $\frac{1}{2}\lambda(S)$. Let $\gamma \in \Gamma(\nu_2,\nu_1)$ be an optimal transport plan between them, with $\| \sfd_\sfX(x,x') \|_{L^p(\gamma)}^p  = W_p(\nu_2,\nu_1)^p$.

We then define
\begin{equation}
\label{eq:eta_construction_lemma}
\eta := \eta_3 + \gamma + (\ii \times \ii)_\sharp \big( \tfrac{1}{2}\,\lambda \mres S \big).
\end{equation}
A direct computation shows that the first marginal of $\eta$ is
\begin{equation*}
\pi^1_\sharp \eta_3 + \pi^1_\sharp \gamma + \tfrac{1}{2}\lambda\mres S = (\mu - \nu_2 - \tfrac{1}{2}\lambda\mres S) + \nu_2 + \tfrac{1}{2}\lambda\mres S = \mu,
\end{equation*}
and analogously $\pi^2_\sharp \eta = \mu'$. Thus $\eta \in \Gamma(\mu,\mu')$.

\uline{Cost estimate.} We bound $\| \sfd_\sfX(x,x') \|_{L^p(\eta)}^p $ in the three terms. For $\eta_3$, since $\eta_3 \leq \eta_0$,
\begin{equation*}
\int \sfd_\sfX(x,x')^p \, \d\eta_3(x,x') \leq \int \sfd_\sfX(x,x')^p \, \d\eta_0(x,x') = W_p(\mu,\mu')^p.
\end{equation*}
The diagonal contribution $(\ii \times \ii)_\sharp(\frac{1}{2}\lambda\mres S)$ gives zero. For $\gamma$, by the triangle inequality in $W_p$, using $\eta_2$ as a coupling between $\nu_2$ and $\frac{1}{2}\lambda\mres S$ and $\eta_1$ as a coupling between $\frac{1}{2}\lambda\mres S$ and $\nu_1$,
\begin{align*}
W_p(\nu_2,\nu_1) \leq W_p(\nu_2, \tfrac{1}{2}\lambda\mres S) + W_p(\tfrac{1}{2}\lambda\mres S, \nu_1) & \leq \|\sfd_\sfX(x,x')\|_{L^p(\eta_2)} + \|\sfd_\sfX(x,x')\|_{L^p(\eta_1)}  \\
& \leq 2 \|\sfd_\sfX(x,x')\|_{L^p(\eta_0)} = 2 W_p(\mu,\mu').
\end{align*}
Combining,
\begin{equation*}
\int \sfd_\sfX(x,x')^p \, \d\eta(x,x') \leq 3 \, W_p(\mu,\mu')^p.
\end{equation*}

\uline{Application of the key estimate.} With~\eqref{eq:regularity_fundamental},
\begin{equation*}
\|\sfd_\sfY(\ff(x,\mu),\ff(x',\mu'))\|_{L^p(\eta)} \leq L \, \|\sfd_\sfX(x,x')\|_{L^p(\eta)} \leq 3 \, L \, W_p(\mu,\mu').
\end{equation*}
The diagonal contribution to the left-hand side gives, in particular,
\begin{equation*}
\int_S \sfd_\sfY(\ff(x,\mu),\ff(x,\mu'))^p \, \tfrac{1}{2}\,\d\lambda(x) \leq 3 \, L^p \, W_p(\mu,\mu')^p.
\end{equation*}

\uline{Localization at $x_0$.} Since $\mu, \mu' \geq \lambda$, we have $\supp(\lambda) \subseteq \supp(\mu) \cap \supp(\mu')$, so in particular $x_0 \in \supp(\mu) \cap \supp(\mu')$. By Corollary~\ref{crl:f_jointly_continuous}, 
$\ff(\cdot,\mu)$ and $\ff(\cdot,\mu')$ are $L$-Lipschitz on the supports of $\mu$ and $\mu'$ respectively. For any $x \in S \cap \supp(\lambda) \subseteq \supp(\mu) \cap \supp(\mu')$ we have
\begin{equation*}
\sfd_\sfY(\ff(x_0,\mu),\ff(x_0,\mu')) \leq \sfd_\sfY(\ff(x,\mu),\ff(x,\mu')) + L\,\sfd_\sfX(x_0,x) + L\,\sfd_\sfX(x_0,x) \leq \sfd_\sfY(\ff(x,\mu),\ff(x,\mu')) + 2Lr.
\end{equation*}
Raising to the $p$-th power and integrating against $\frac{1}{2}\lambda \mres S$ (which has mass $\frac{1}{2}\lambda(S) > 0$), then dividing by this mass:
\begin{equation*}
\sfd_\sfY(\ff(x_0,\mu),\ff(x_0,\mu'))^p \leq 2^{p-1} \left( \frac{1}{\frac{1}{2}\lambda(S)} \int_S \sfd_\sfY(\ff(x,\mu),\ff(x,\mu'))^p \, \tfrac{1}{2}\,\d\lambda(x) + (2Lr)^p \right).
\end{equation*}
Combining with the previous bound,
\begin{equation*}
\sfd_\sfY(\ff(x_0,\mu),\ff(x_0,\mu'))^p \leq C(p) \, L^p \left( r^p + \frac{W_p(\mu,\mu')^p}{\lambda(B(x_0,r))} \right),
\end{equation*}
which is the claimed estimate.
\end{proof}

The right-hand side of the estimate in Lemma~\ref{lm:uniform_continuity_supp} depends on $r$. To obtain a quantitative rate as $W_p(\mu,\mu') \to 0$ one needs an assumption on the local behavior of $\lambda(B(\cdot,r))$ as $r \to 0$, uniform in $x \in \supp(\lambda)$. A typical instance is the following: one should think $\sfX$ is a bounded convex (or regular) subset with nonempty interior in an Euclidean space of dimension $k$ and $\lambda$ is (a scalar multiple of) the Lebesgue measure.
Another interesting example is provided by a compact Riemannian manifold $\sfX$ endowed with the Riemannian volume measure $\lambda$.

\begin{corollary}
\label{crl:holder_uniform}
Under the assumptions of Lemma~\ref{lm:uniform_continuity_supp}, assume moreover that there exist $c > 0$ and $k > 0$ such that
\begin{equation*}
\lambda(B(x,r)) \geq c \, r^k \qquad \text{for every } x \in \supp(\lambda) \text{ and every } r \in (0,1].
\end{equation*}
Then for every $\mu, \mu' \in \cP_p(\sfX)$ with $\mu,\mu' \geq \lambda$,
\begin{equation*}
\sup_{x \in \supp(\lambda)} \sfd_\sfY(\ff(x,\mu),\ff(x,\mu')) \leq C(p,k,c) \, L \, W_p(\mu,\mu')^{p/(p+k)},
\end{equation*}
provided $W_p(\mu,\mu')$ is small enough. In particular, $\ff(\cdot,\mu_n) \to \ff(\cdot,\mu)$ uniformly on $\supp(\lambda)$ whenever $\mu_n \to \mu$ in $\cP_p(\sfX)$ with $\mu_n, \mu \geq \lambda$.
\end{corollary}

\begin{proof}
Optimize the estimate of Lemma~\ref{lm:uniform_continuity_supp} in $r$: substituting $\lambda(B(x,r)) \geq c r^k$ and choosing $r := (W_p(\mu,\mu')^p/c)^{1/(p+k)}$ makes the two terms $r^p$ and $W_p(\mu,\mu')^p/(c r^k)$ comparable, and yields the Hölder estimate with exponent $p/(p+k)$ after taking the $p$-th root.
\end{proof}

\subsection{Universal approximation}
\label{sec:approximation}

We conclude this section with the proof of the approximation result (Theorem~\ref{thm:approximation}): any Lipschitz non-splitting transformation $F$ can be approximated, in Wasserstein distance and uniformly in $\mu$, by the push-forward by elements of any sufficiently rich family $\cA$ of continuous mappings $\sfX \times \cP(\sfX) \to \sfY$.

We emphasize why the approximation is formulated in Wasserstein distance at the level of the map $F$, rather than as a uniform approximation of $\ff$ itself by elements of $\cA$. Example~\ref{ex:lift_discontinuous_supp} exhibits an $L$-Lipschitz non-splitting transformation $F$ whose transport representative $\ff$ cannot be extended to a jointly continuous mapping on $\sfX \times \cP_p(\sfX)$: in that example, $\ff$ has an unavoidable discontinuity at $(1,\delta_0)$, with oscillation bounded below by $2$. If $\cA$ is a family of jointly continuous mappings (as is the case for transformer architectures and, more generally, for any natural parametric family of neural networks), no element of the uniform closure of $\cA$ can approximate $\ff$ uniformly on $\sfX \times \cP_p(\sfX)$. The Wasserstein formulation of Theorem~\ref{thm:approximation} circumvents this obstruction: since the equality $F(\mu) = \ff(\cdot,\mu)_\sharp \mu$ only depends on the values of $\ff(\cdot,\mu)$ on $\supp(\mu)$, approximating $F(\mu)$ in Wasserstein distance is a genuinely weaker (but natural) requirement which is compatible with the pointwise discontinuities of $\ff$ at points $(x,\mu)$ with $x \notin \supp(\mu)$.

\begin{proof}[\textbf{Proof of Theorem~\ref{thm:approximation}}]
We fix $\varepsilon > 0$.
Let $\ff$ be the transport representative of $F$ given by Theorem~\ref{thm:transport_lift_lipschitz}, satisfying the joint continuity property of Corollary~\ref{crl:f_jointly_continuous}. Since $\sfX$ is a Polish space, we can fix a probability measure $\lambda \in \cP(\sfX)$ with full support: for instance, if $\{x_n\}_{n \geq 1}$ is a countable dense subset of $\sfX$, the measure $\lambda := \sum_{n \geq 1} 2^{-n}\,\delta_{x_n}$ has $\supp(\lambda) = \sfX$.

For a function $\alpha : \cP_p(\sfX) \to (0,1]$ we set
% For every $\alpha \in (0,1]$ and $\mu \in \cP_p(\sfX)$ we set
\begin{equation*}
\mu_\alpha := (1-\alpha(\mu)) \mu + (\mu) \lambda, \qquad \ff_\alpha(x,\mu) := \ff(x,\mu_\alpha).
\end{equation*}
Note that $\mu_\alpha \geq \alpha(\mu)\lambda$, and since $\supp(\lambda) = \sfX$, we have $\supp(\mu_\alpha) = \sfX$ for every $\mu \in \cP(\sfX)$. Hence $\ff(\cdot,\mu_\alpha) = \ff_\alpha(\cdot,\mu)$ is $L$-Lipschitz on all of $\sfX$ by Corollary~\ref{crl:f_jointly_continuous}.

\uline{Choice of $\alpha$.}
By convexity of $W^p_p$ we have $W_p^p(\mu_\alpha,\mu) \leq \alpha(\mu)W_p^p(\mu,\lambda)$. We choose 
\begin{equation*}
    \alpha(\mu) = \min \left( 1, \frac{\varepsilon}{3 L W_p(\mu,\lambda)}   \right)^p,
\end{equation*}
which makes it a continuous function on $\cP_p(\sfX)$, and which also ensures that
\begin{equation}
\label{eq:aux_mu_mutheta}
W_p(\mu,\mu_\alpha) \leq \frac{\varepsilon}{3L} \qquad \text{for all } \mu \in \cP_p(\sfX).
\end{equation}

\uline{Continuity and Lipschitzianity of $\ff_\alpha$.} We claim that $\ff_\alpha : \sfX \times \cP(\sfX) \to \sfY$ is jointly continuous and $L$-Lipschitz in the first variable. The $L$-Lipschitzianity in $x$ has already been noted. For joint continuity, let $(x_n,\mu_n) \to (x,\mu)$ in $\sfX \times \cP(\sfX)$. Then $(\mu_n)_\alpha \to \mu_\alpha$ in $\cP(\sfX)$ (because $\alpha$ is continuous), $x \in \sfX = \supp(\mu_\alpha)$, and $x_n \in \sfX = \supp((\mu_n)_\alpha)$ trivially; hence by Corollary~\ref{crl:f_jointly_continuous}, $\ff_\alpha(x_n,\mu_n) = \ff(x_n,(\mu_n)_\alpha) \to \ff(x,\mu_\alpha) = \ff_\alpha(x,\mu)$.

\uline{Approximation by $\cA$.} By the assumption on $\cA$, since $\ff_\alpha$ is jointly continuous and $L$-Lipschitz in the first variable, $\ff_\alpha$ belongs to the uniform closure of $\cA$. Thus there exists $\gg \in \cA$ such that
\begin{equation*}
\sup_{(x,\mu) \in \sfX \times \cP(\sfX)} \sfd_\sfY(\gg(x,\mu),\ff_\alpha(x,\mu)) \leq \frac{\varepsilon}{3}.
\end{equation*}

\uline{Conclusion.} For every $\mu \in \cP(\sfX)$, by the triangle inequality in $W_p$,
\begin{equation*}
W_p(\gg(\cdot,\mu)_\sharp \mu, F(\mu)) \leq W_p(\gg(\cdot,\mu)_\sharp \mu, \ff_\alpha(\cdot,\mu)_\sharp \mu) + W_p(\ff_\alpha(\cdot,\mu)_\sharp \mu, F(\mu_\alpha)) + W_p(F(\mu_\alpha), F(\mu)).
\end{equation*}
The first term is bounded by~\eqref{eq:Wp_Lp} and the previous step:
\begin{equation*}
W_p(\gg(\cdot,\mu)_\sharp \mu, \ff_\alpha(\cdot,\mu)_\sharp \mu) \leq \|\sfd_\sfY(\gg(\cdot,\mu),\ff_\alpha(\cdot,\mu))\|_{L^p(\mu)} \leq \frac{\varepsilon}{3}.
\end{equation*}
For the second term, observe that $\ff_\alpha(\cdot,\mu)_\sharp \mu = \ff(\cdot,\mu_\alpha)_\sharp \mu$ and $F(\mu_\alpha) = \ff(\cdot,\mu_\alpha)_\sharp \mu_\alpha$. By the standard estimate~\eqref{eq:lift_lipschitz} and our choice of $\alpha$ which guarantees~\eqref{eq:aux_mu_mutheta},
\begin{equation*}
W_p(\ff(\cdot,\mu_\alpha)_\sharp \mu, \ff(\cdot,\mu_\alpha)_\sharp \mu_\alpha) \leq L W_p(\mu,\mu_\alpha) \leq \frac{\varepsilon}{3}
\end{equation*}
The third term is bounded with the help of~\eqref{eq:aux_mu_mutheta} as well: $W_p(F(\mu_\alpha),F(\mu)) \leq \varepsilon/3$. Summing, $W_p(\gg(\cdot,\mu)_\sharp \mu, F(\mu)) \leq \varepsilon$, as claimed.
\end{proof}

\section{Lagrangian representations of continuous transformations}
\label{sec:lagrangian}

We discuss the notion of Lagrangian representation, namely when we represent $F$ as a transformation of random variables. This concept is linked to the Lions lift \cite{cardaliaguet2010notes,Carmona-Delarue18} and to the study of dissipative evolutions in Wasserstein spaces \cite{cavagnari2023lagrangian,CSS25}. We prove that, for continuous transformations $F$, having a continuous transport representation (in the sense of Definition~\ref{def:meas_continuous_lipschitz_TLp}) is equivalent to having a continuous Lagrangian representation invariant by measure preserving isomorphisms.

In the sequel we denote by $(\Omega,\frB,\P)$ 
a standard Borel space, for instance $\Omega = [0,1]$ and $\P$ is the Lebesgue measure.

We are interested by a representation of $F$ as a transformation of laws of random variables: we look for $\hat \frF$ which transforms a random variable $X \in L^p(\Omega;\sfX)$ into another random variable $Y = \hat \frF(X) \in L^p(\Omega;\sfY)$ such that the law of $Y$ is the image of the law of $X$ by $F$. As understood in the study of Lagrangian representations of dissipative evolutions in Wasserstein spaces, to encode that $F$ is non-splitting, it is natural to require $\hat \frF$ to be invariant by measure preserving isomorphisms.  

\begin{definition}
A map $F : \cP_p(\sfX) \to \cP_p(\sfY)$ has a \emph{Lagrangian representative} if there exists $\hat \frF : L^p(\Omega;\sfX) \to L^p(\Omega;\sfY)$ such that, if $Y = \hat \frF(X)$, then $Y_\sharp \P = F( X_\sharp \P)$. 

We say that the Lagrangian representative is invariant by measure-preserving isomorphism (m.p.i.) if, for every $X \in L^p(\Omega;\sfX)$ and every Borel $g : \Omega \to \Omega$ which leaves $\P$ invariant, $\hat \frF(X \circ g) = \hat \frF(X) \circ g$.
\end{definition}

If $F$ admits a transport representative $\frF$, then it admits a Lagrangian representation invariant by m.p.i. Indeed, with $\frF(\mu) = (\mu, \ff(\cdot,\mu))$ we define $\hat \frF$ by 
\begin{equation}
\label{eq:def_lagrangian_from_transport}
\hat \frF(X) : \omega \mapsto \ff(X(\omega), X_\sharp \P).
\end{equation}
This is a Lagrangian representative of $F$: if $Y = \hat \frF(X)$ then $Y_\sharp \P = \ff(X(\cdot), X_\sharp \P)_\sharp \P = \ff(\cdot,X_\sharp \P)_\sharp (X_\sharp \P) = F(X_\sharp \P)$. Moreover, as $(X \circ g)_\sharp \P = X_\sharp \P$ for every m.p.i. $g$, this Lagrangian representative is invariant by m.p.i. For continuous transformations, we show that the converse actually holds.

\begin{theorem}
\label{thm:Lagrangian}
A map $F : \cP_p(\sfX) \to \cP_p(\sfY)$ has a continuous transport representative (in the sense of Definition~\ref{def:meas_continuous_lipschitz_TLp}) if and only if it has a continuous Lagrangian representative invariant by m.p.i. 
\end{theorem}

In this case the map $F$ is necessarily continuous. As we show below in Example~\ref{ex:law_invariance_not_enough}, only imposing $\hat{\frF}$ to be law invariant is not enough to have a transport representative. 
For the proof of this theorem we will use the following classical lemma which can be traced back to \cite{brenier2003approximation}, see \cite[Lemma 6.4]{cardaliaguet2010notes} or \cite[Corollary 3.16]{CSS25} for a proof. 

\begin{lemma}
\label{lm:mpi_approx}
Assume $X,X'$ are two elements of $L^p(\Omega;\sfX)$ such that $X_\sharp \P = X'_\sharp \P$. Then there exists a sequence $(g_n)_{n \in \N}$ of m.p.i. such that $X \circ g_n$ converges to $X'$ in $L^p(\Omega;\sfX)$.
\end{lemma}

\begin{proof}[Proof of Theorem~\ref{thm:Lagrangian}]
\uline{Direct implication}.
Assume first that $\frF$ is a continuous transport representative of $F$. Then recall that we know that $\hat \frF$, defined in~\eqref{eq:def_lagrangian_from_transport}, is a Lagrangian representative invariant by m.p.i. Thus we only have to check that $\hat \frF$ is continuous. 

Let $(X_n)_{n \in \N}$ a sequence converging to $X$ in $L^p(\Omega;\sfX)$, we write $Y_n = \hat \frF(X_n)$, $Y = \hat \frF(X)$ as well as $\mu_n = (X_n)_\sharp \P$ and $\mu = X_\sharp \P$.
As $X_n$ converges to $X$ in $L^p(\Omega;\sfX)$, we know that $\mu_n$ converges to $\mu$ in $\cP_p(\sfX)$. Thus by continuity of $\frF$, $(\mu_n,\ff(\cdot,\mu_n))$ converges to $(\mu,\ff(\cdot,\mu))$ in $\rmT\rmL_p(\sfX;\sfY)$. As moreover $(X_n,X)_\sharp \P$ is stagnating because $X_n$ converges to $X$ in $L^p(\Omega;\sfX)$, we deduce that 
\begin{equation*}
\|\sfd_\sfY(Y_n,Y)\|_{L^p(\Omega)} = \| \sfd_\sfY( \ff(X_n,\mu_n), \ff(X,\mu)  ) \|_{L^p(\Omega)} = \| \sfd_\sfY( \ff(x,\mu_n), \ff(x',\mu)  ) \|_{L^p((X_n,X)_\sharp \P)}
\end{equation*}
also converges to $0$, see Proposition~\ref{prop:char_conv_TLp}(3). It exactly means that $Y_n$ converges to $Y$ as $n \to + \infty$ in $L^p(\Omega;\sfY)$.

\uline{Converse implication}.
Conversely, assume that $\hat \frF$ is a continuous Lagrangian representative invariant by m.p.i. Let us consider the subset of $\cP_p(\sfX \times \sfY)$ defined by 
\begin{equation*}
G = \{ (X,\hat \frF(X))_\sharp \P \ : \ X \in L^p(\Omega;\sfX) \},
\end{equation*}
which we will prove gives us the graph of the transport representative $\frF$.

First we claim that, if $(X,Y)_\sharp \P \in G$, then necessarily $Y = \hat{\frF}(X)$. Indeed, in such a case by definition $(X,Y)_\sharp \P = (X',\hat \frF(X'))_\sharp \P$ for some $X' \in L^p(\Omega;\sfX)$. By Lemma~\ref{lm:mpi_approx} we can find a sequence $g_n$ of m.p.i. such that $(X' \circ g_n, \hat \frF(X') \circ g_n)$ converges in $L^p(\Omega;\sfX \times \sfY)$ to $(X,Y)$. In particular $X' \circ g_n$ converges to $X$, while $\hat \frF(X') \circ g_n = \hat \frF(X' \circ g_n)$ converges to $Y$. By continuity of $\hat \frF$ we obtain $Y = \hat{\frF}(X)$. 

Second, we claim that any $\gamma \in G$ belongs to $\cPg_p(\sfX \times \sfY)$. Indeed, let $\gamma \in G$ and $\sigma \in \Gamma(\gamma,\gamma)$. That means that we can find random variables $X,Y,X',Y'$ with $\sigma = (X,Y,X',Y')_\sharp \P$ and $(X,Y)_\sharp \P = (X',Y')_\sharp \P = \gamma$. From the first claim we deduce $Y= \hat \frF(X)$ and $Y'= \hat \frF(X')$. Thus if $X = X'$ almost surely, then clearly $Y=Y'$ almost surely which implies the claim with Lemma~\ref{lm:char_cpg}.

Third, take $(X_n,\hat \frF(X_n))_\sharp \P = \gamma_n$ a sequence in $G$ and assume that it converges to $\gamma \in \cP_p(\sfX \times \sfY)$. With $\mu$ the first marginal of $\gamma$ and $X'$ such that $X'_\sharp \P = \mu$, we will show that $\gamma_n$ converges to $(X',\hat \frF (X'))_\sharp \P$. Indeed, as ${X_n}_\sharp \P$ converges to $X'_\sharp \P$, by Lemma~\ref{lm:mpi_approx} we can find a sequence $g_n$ of m.p.i. such that $X_n \circ g_n$ converges to $X'$ in $L^p(\Omega;\sfX)$. By continuity of $\hat \frF$, then $\hat \frF(X_n \circ g_n)$ converges to $\hat \frF (X')$. As $\hat \frF$ is invariant by m.p.i., 
\begin{equation*}
\gamma_n = (X_n,\hat \frF(X_n))_\sharp \P = (X_n \circ g_n,\hat \frF(X_n) \circ g_n)_\sharp \P = (X_n \circ g_n,\hat \frF(X_n \circ g_n))_\sharp \P.   
\end{equation*}
Taking the limit $n \to + \infty$, in the right hand side we see it converges to $(X',\hat \frF (X'))_\sharp \P = \gamma$. Thus we have our claim. This claim implies that $G$ is closed. Moreover, taking for $(X_n,\hat \frF(X_n))_\sharp \P$ a constant sequence, as the limit $\gamma$ depends only on $\mu$, we conclude that there exists one element $\gamma \in G$ with $\pi^1_\sharp \gamma = \mu$.

We now are in position to define $\frF$: for any $\mu \in \cP_p(\sfX)$, there exists a unique $\gamma \in G$ with $\pi^1_\sharp \gamma = \mu$ (and automatically $\pi^2_\sharp \gamma = F(\mu)$), and this $\gamma$ belongs to $\cPg_p(\sfX \times \sfY)$. Thus we can define $\frF(\mu) = \iota^{-1}(\gamma)$. 
% Thus it can be written $\gamma = (\ii \times \ff(\cdot,\mu))_\sharp \mu$, and we define $\frF(\mu) = (\mu, f)$. 
If $\mu_n \to \mu$ as $n \to + \infty$, we write $\gamma_n$ for the unique element of $G$ with first marginal $\mu_n$, it belongs to $\Gamma(\mu_n,F(\mu_n))$. As $F(\mu_n) \to F(\mu)$, up to extraction, $(\gamma_n)_{n \in \N}$ converges to an element of $\Gamma(\mu,F(\mu))$. Moreover $G$ is closed thus $(\gamma_n)_{n \in \N}$ necessarily converges to $\gamma$, which means $\frF(\mu_n)$ converges to $\frF(\mu)$ in $\rmT\rmL_p(\sfX;\sfY)$.
\end{proof}

\begin{example}
\label{ex:law_invariance_not_enough}
Instead of imposing $\hat{\frF}$ to be invariant by m.p.i., one could impose the weaker property of law invariance: that $X_\sharp \P = X'_\sharp \P$ implies $\hat{\frF}(X)_\sharp \P = \hat{\frF}(X')_\sharp \P$. However this invariance does not imply the existence of a transport representative.

Indeed take $Y_0 \in L^p(\Omega;\sfY)$ non-constant fixed. Then the constant map $\hat{\frF} : X \mapsto Y_0$ is law invariant and continuous. However it does not have a transport representative: if $X$ is a constant random variable, with law $X_\sharp \P= \delta_x$ for some $x \in \sfX$, then $\hat{\frF}(X)_\sharp \P = {Y_0}_\sharp \P$ which is not a Dirac mass by assumption, thus cannot be the image of $X_\sharp \P$ by a transport map. 
\end{example}

\bibliographystyle{siam}
\bibliography{bibliografia}

\begin{thebibliography}{10}

\bibitem{alberti2023sumformer}
{\sc S.~Alberti, N.~Dern, L.~Thesing, and G.~Kutyniok}, {\em {Sumformer:
  Universal approximation for efficient transformers}}, in Topological,
  Algebraic and Geometric Learning Workshops 2023, PMLR, 2023, pp.~72--86.

\bibitem{aliprantis2006infinite}
{\sc C.~D. Aliprantis and K.~C. Border}, {\em Infinite dimensional analysis: a
  hitchhiker's guide}, Springer Science \& Business Media, 2006.

\bibitem{AGS}
{\sc L.~Ambrosio, N.~Gigli, and G.~Savar{\'e}}, {\em {Gradient flows in metric
  spaces and in the space of probability measures}}, Lectures in Mathematics
  ETH Zurich, Birkh{\"a}user Verlag, 2008.

\bibitem{ambrosio2004topics}
{\sc L.~Ambrosio and P.~Tilli}, {\em Topics on analysis in metric spaces},
  Oxford University Press, 2004.

\bibitem{bergin1999continuity}
{\sc J.~Bergin}, {\em On the continuity of correspondences on sets of measures
  with restricted marginals}, Econom. Theory, 13 (1999), pp.~471--481.

\bibitem{bogachev2007measure}
{\sc V.~I. Bogachev}, {\em Measure theory}, vol.~2, Springer, 2007.

\bibitem{brenier2003approximation}
{\sc Y.~Brenier and W.~Gangbo}, {\em {Approximation of maps by
  diffeomorphisms}}, Calc. Var. Partial Differential Equations, 16 (2003),
  pp.~147--164.

\bibitem{cardaliaguet2010notes}
{\sc P.~Cardaliaguet}, {\em Notes on mean field games}, tech. rep., 2010.

\bibitem{Carmona-Delarue18}
{\sc R.~Carmona and F.~c. Delarue}, {\em Probabilistic theory of mean field
  games with applications. {I}}, vol.~83 of Probability Theory and Stochastic
  Modelling, Springer, Cham, 2018.
\newblock Mean field FBSDEs, control, and games.

\bibitem{cavagnari2023lagrangian}
{\sc G.~Cavagnari, G.~Savar{\'e}, and G.~E. Sodini}, {\em {A Lagrangian
  approach to totally dissipative evolutions in Wasserstein spaces}}, arXiv
  preprint arXiv:2305.05211,  (2023).

\bibitem{CSS25}
{\sc G.~Cavagnari, G.~Savar\'e, and G.~E. Sodini}, {\em Extension of monotone
  operators and {L}ipschitz maps invariant for a group of isometries}, Canad.
  J. Math., 77 (2025), pp.~149--186.

\bibitem{fornasier2023density}
{\sc M.~Fornasier, G.~Savar{\'e}, and G.~E. Sodini}, {\em {Density of
  subalgebras of Lipschitz functions in metric Sobolev spaces and applications
  to Wasserstein Sobolev spaces}}, J. Funct. Anal., 285 (2023), p.~110153.

\bibitem{furuya2024transformers}
{\sc T.~Furuya, M.~V. de~Hoop, and G.~Peyr{\'e}}, {\em {Transformers are
  universal in-context learners}}, arXiv preprint arXiv:2408.01367,  (2024).

\bibitem{GarciaTrillos-Slepcev16}
{\sc N.~Garc\'ia~Trillos and D.~Slep\v{c}ev}, {\em Continuum limit of total
  variation on point clouds}, Arch. Ration. Mech. Anal., 220 (2016),
  pp.~193--241.

\bibitem{trillos2018variational}
\leavevmode\vrule height 2pt depth -1.6pt width 23pt, {\em {A variational
  approach to the consistency of spectral clustering}}, Appl. Comput. Harmon.
  Anal., 45 (2018), pp.~239--281.

\bibitem{garey1979computers}
{\sc M.~R. Garey and D.~S. Johnson}, {\em {Computers and Intractability: a
  Guide to the Theory of NP-Completeness}}, W.H. freeman New York, 1979.

\bibitem{geshkovski2025mathematical}
{\sc B.~Geshkovski, C.~Letrouit, Y.~Polyanskiy, and P.~Rigollet}, {\em A
  mathematical perspective on transformers}, Bull. Amer. Math. Soc., 62 (2025),
  pp.~427--479.

\bibitem{geshkovski2024measure}
{\sc B.~Geshkovski, P.~Rigollet, and D.~Ruiz-Balet}, {\em Measure-to-measure
  interpolation using transformers}, arXiv preprint arXiv:2411.04551,  (2024).

\bibitem{ghossoub2021continuity}
{\sc M.~Ghossoub and D.~Saunders}, {\em {On the continuity of the feasible set
  mapping in optimal transport}}, Econ. Theory Bull., 9 (2021), pp.~113--117.

\bibitem{kallenberg2017random}
{\sc O.~Kallenberg}, {\em Random measures, theory and applications}, vol.~1,
  Springer, 2017.

\bibitem{pinkus1999approximation}
{\sc A.~Pinkus}, {\em {Approximation theory of the MLP model in neural
  networks}}, Acta Numer., 8 (1999), pp.~143--195.

\bibitem{sander2022sinkformers}
{\sc M.~E. Sander, P.~Ablin, M.~Blondel, and G.~Peyr{\'e}}, {\em {Sinkformers:
  Transformers with doubly stochastic attention}}, in International Conference
  on Artificial Intelligence and Statistics, PMLR, 2022, pp.~3515--3530.

\bibitem{thorpe2017transportation}
{\sc M.~Thorpe, S.~Park, S.~Kolouri, G.~K. Rohde, and D.~Slep\v{c}ev}, {\em A
  transportation {$L^p$} distance for signal analysis}, J. Math. Imaging
  Vision, 59 (2017), pp.~187--210.

\bibitem{vaswani2017attention}
{\sc A.~Vaswani, N.~Shazeer, N.~Parmar, J.~Uszkoreit, L.~Jones, A.~N. Gomez,
  L.~Kaiser, and I.~Polosukhin}, {\em Attention is all you need}, Advances in
  neural information processing systems, 30 (2017).

\bibitem{villani2009optimal}
{\sc C.~Villani}, {\em Optimal transport: old and new}, vol.~338, Springer,
  2009.

\bibitem{vuckovic2020mathematical}
{\sc J.~Vuckovic, A.~Baratin, and R.~Tachet~des Combes}, {\em {A mathematical
  theory of attention}}, arXiv preprint arXiv:2007.02876,  (2020).

\bibitem{yun2019transformers}
{\sc C.~Yun, S.~Bhojanapalli, A.~S. Rawat, S.~J. Reddi, and S.~Kumar}, {\em
  {Are transformers universal approximators of sequence-to-sequence
  functions?}}, arXiv preprint arXiv:1912.10077,  (2019).

\end{thebibliography}

\end{document}